\documentclass[leqno,12pt]{amsart} 
\usepackage{amssymb, amsmath, color, amstext, tikz}
\usepackage[english]{babel}
\usepackage{appendix}
\newtheorem{letterthm}{Theorem}

\newtheorem{theo}{Theorem}[section]

\newtheorem{prop}[theo]{Proposition}
\theoremstyle{definition} 
\newtheorem{defi}[theo]{Definition}
\newtheorem{remark}[theo]{Remark}

\newcommand{\AN}{almost normal}
\newcommand{\BGPA}{bipartite graph planar algebra}
\newcommand{\biha}{Bisch-Haagerup}

\newcommand{\CE}{conditional expectation}
\newcommand{\CrP}{crossed product}
\newcommand{\copo}{completely positive}

\newcommand{\FPS}{fixed point space}
\newcommand{\FPPA}{fixed point planar algebra}
\newcommand{\HS}{Hilbert space}
\newcommand{\Hepa}{Hecke pair}
\newcommand{\IFF}{if and only if}

\newcommand{\nofa}{normal faithful}
\newcommand{\NFS}{normal faithful state}
\newcommand{\NFTS}{normal faithful tracial state}
\newcommand{\ONB}{orthonormal basis}
\newcommand{\PA}{planar algebra}
\newcommand{\pode}{positive definite}

\newcommand{\SF}{subfactor}
\newcommand{\SPA}{subfactor planar algebra}
\newcommand{\stre}{$*$-representation}
\newcommand{\SEA}{symmetric enveloping algebra}
\newcommand{\SEI}{symmetric enveloping inclusion}
\newcommand{\SI}{standard invariant}
\newcommand{\SMU}{system of matrix units}
\newcommand{\SOT}{strong operator topology}
\newcommand{\SR}{system of representatives}
\newcommand{\TFAT}{The following assertions are true.}
\newcommand{\TLJ}{Temperley-Lieb-Jones}
\newcommand{\tII}{type II$_1$}
\newcommand{\VNA}{von Neumann algebra}
\newcommand{\WRT}{with respect to}
\newcommand{\Aut}{\text{Aut}}
\newcommand{\AutG}{\text{Aut}(\Ga,\mu)}
\newcommand{\AutP}{\text{Aut}(\Pl)}
\newcommand{\act}{\curvearrowright}
\newcommand{\CAGH}{\C[A;G,H]}
\newcommand{\C}{\mathbb C}
\newcommand{\ep}{\epsilon}
\newcommand{\gO}{\geqslant 0}
\newcommand{\ga}{\gamma}
\newcommand{\Ga}{\Gamma}
\newcommand{\Gogo}{ {G_{o,go} } }
\newcommand{\GGG}{\langle G_o\backslash G/G_o\rangle}
\newcommand{\GrkPb}{Gr_k\Pl\boxtimes Gr_k\Pl}
\newcommand{\J}{\mathcal J}

\newcommand{\lGHr}{\langle G/H\rangle}

\newcommand{\lGGor}{\langle G/G_o\rangle}
\newcommand{\GGo}{\langle G/G_o\rangle}
\newcommand{\lGGogor}{ \langle G / \Gogo \rangle }

\newcommand{\loriar}{\longrightarrow}
\newcommand{\mH}{\mathcal H}
\newcommand{\mN}{\mathcal N}
\newcommand{\mM}{\mathcal M}
\newcommand{\NM}{{N\subset M}}

\newcommand{\ootimes}{\bar\otimes}
\newcommand{\op}{\text{op}}
\newcommand{\cP}{\mathcal P}
\newcommand{\Pl}{\mathcal P}
\newcommand{\Ql}{\mathcal Q}
\newcommand{\R}{\mathbb R}
\newcommand{\spann}{\text{Span}}
\newcommand{\TS}{{T\subset S}}
\newcommand{\UP}{U(\Pl_1^+)}
\newcommand{\Ve}{V^\epsilon}
\newcommand{\vNAGH}{\text{vN}[A;G,H]}
\newcommand{\vNTGGo}{\text{vN}[T(o);G,G_o]}

\newcommand{\WOT}{weak operator topology}

\newcommand{\diagxdag}{\begin{tikzpicture}[baseline = .425cm]\draw (.25,.25)--(.75,.25)--(.75,.75)--(.25,.75)--(.25,.25);\draw (.5,0)--(.5,.25);\draw (.5,.75)--(.5,1);\draw (0,.5)--(.25,.5);\draw (.75,.5)--(1,.5);\node at (.5,.5) { $x^*$ };\node at (.85,.85) {\tiny\$};\end{tikzpicture}}

\newcommand{\diagthree}{\begin{tikzpicture}[baseline=-.125cm]\draw (-.25,-.25)--(.25,-.25)--(.25,.25)--(-.25,.25)--(-.25,-.25);\draw (-.25,.5)--(.25,.5);\draw (-.25,-.5)--(.25,-.5);\draw (-.25,-.125) arc (90 : 270 : .1875);\draw (-.25,.5) arc (90 : 270 : .1875);\draw (.25,.125) arc (-90 : 90 : .1875);\draw (.25,-.5) arc (-90 : 90 : .1875);\node at (0,0) { $x$ };\node at (.5,.375) {{ \scriptsize{k} }};\node at (.5,-.375) {{ \scriptsize{k} }};\end{tikzpicture}}

\newcommand{\diagtaul}{\begin{tikzpicture}[baseline=-.1cm]\draw (-.25,-.25)--(.25,-.25)--(.25,.25)--(-.25,.25)--(-.25,-.25);\draw (0,.25) arc (0 : 180 : .25);\draw (-.5,-.25) arc (-180 : 0 : .25);\node at (0,0) { $x$ };\draw (-.5,-.25)--(-.5,.25);\end{tikzpicture}}
\newcommand{\diagtaur}{\begin{tikzpicture}[baseline=-.1cm]\draw (-.25,-.25)--(.25,-.25)--(.25,.25)--(-.25,.25)--(-.25,-.25);\draw (.5,.25) arc (0 : 180 : .25);\draw (0,-.25) arc (-180 : 0 : .25);\node at (0,0) { $x$ };\draw (.5,-.25)--(.5,.25);\end{tikzpicture}}
\newcommand{\diagEP}{\frac{1}{\delta}\begin{tikzpicture}[baseline=-.1cm]\draw (-.25,-.25)--(.25,-.25)--(.25,.25)--(-.25,.25)--(-.25,-.25);\draw (.625,.25) arc (0 : 180 : .25);\draw (.125,-.25) arc (-180 : 0 : .25);\node at (0,0) { $x$ };\draw (.625,-.25)--(.625,.25);\draw (-.125,.25)--(-.125,.5);\draw (-.125,-.25)--(-.125,-.5);\node at (-.25,.5) {{ \tiny n }};\node at (-.25,-.5) {{ \tiny n }};\end{tikzpicture}}

\newcommand{\diagpkv}{p_v=\begin{tikzpicture}[baseline=-.1cm]\draw (-.25,-.25)--(.25,-.25)--(.25,.25)--(-.25,.25)--(-.25,-.25);\draw (-.5,-.5)--(.5,-.5);\draw (-.5,.5)--(.5,.5);\node at (0,.6) {{ \scriptsize{k} }};\node at (0,-.6) {{ \scriptsize{k} }};\node at (0,0) { $e_v$ };\end{tikzpicture}=j_k(0,0)(e_v)\in S_k.}
\newcommand{\diagxl}{x_l=\begin{tikzpicture}[baseline=-.1cm]\draw (-.5,-.25)--(.5,-.25)--(.5,.25)--(-.5,.25)--(-.5,-.25);\draw (0,.25)--(0,.5);\draw (0,-.25)--(0,-.5);\node at (0,0) { $e_{a,b}$ };\node at (.15,.5) {{ \scriptsize{2n} }};\node at (.15,-.5) {{ \scriptsize{2m} }};\end{tikzpicture}=\begin{tikzpicture}[baseline=-.1cm]\draw (.75,-.25)--(1.25,-.25)--(1.25,.25)--(.75,.25)--(.75,-.25);\draw (-.75,-.25)--(-1.25,-.25)--(-1.25,.25)--(-.75,.25)--(-.75,-.25);\draw (-.5,-.25)--(.5,-.25)--(.5,.25)--(-.5,.25)--(-.5,-.25);\draw (0,.25)--(0,.5);\draw (0,-.25)--(0,-.5);\node at (0,0) { $e_{a,b}$ };\node at (.15,.5) {{ \scriptsize{2n} }};\node at (.15,-.5) {{ \scriptsize{2m} }};\node at (-1,0) { $e_v$ };\node at (1,0) { $e_v$ };\end{tikzpicture}=\begin{tikzpicture}[baseline=-.1cm]\draw (-.5,.125)--(.5,.125)--(.5,.625)--(-.5,.625)--(-.5,.125);\draw (-.5,-.125)--(.5,-.125)--(.5,-.625)--(-.5,-.625)--(-.5,-.125);\draw (0,.625)--(0,.875);\draw (0,-.625)--(0,-.875);\node at (.15,.875) {{ \scriptsize{2n} }};\node at (.15,-.875) {{ \scriptsize{2m} }};\node at (0,.375) { $e_{a_1,b_1}$ };\node at (0,-.375) { $e_{a_2,b_2}$ };\end{tikzpicture}=x_{l_1}\otimes x_{l_2}.}

\newcommand{\diagWtn}{\begin{tikzpicture}[baseline=0cm]\draw (-.5,.-.25)--(.5,-.25)--(.5,.25)--(-.5,.25)--(-.5,-.25);\draw (-.25,.25)--(-.25,.5);\draw (.35,.25) arc (0 : 180 : .175);\node at (-.55,.4) {\tiny{2n-2}};\node at (0,0) { $x$ };\end{tikzpicture}=
\begin{tikzpicture}[baseline=0cm]\draw (-.5,.-.25)--(.5,-.25)--(.5,.25)--(-.5,.25)--(-.5,-.25);\draw (.25,.25)--(.25,.5);\draw (0,.25) arc (0 : 180 : .175);\node at (.55,.4) {\tiny{2n-2}};\node at (0,0) { $x$ };\end{tikzpicture}=0}

\newcommand{\diagWbm}{\begin{tikzpicture}[baseline=0cm]\draw (-.5,.-.25)--(.5,-.25)--(.5,.25)--(-.5,.25)--(-.5,-.25);\draw (-.25,-.25)--(-.25,-.5);\draw (.35,-.25) arc (0 : -180 : .175);\node at (-.55,-.4) {\tiny{2m-2}};\node at (0,0) { $x$ };\end{tikzpicture}=
\begin{tikzpicture}[baseline=0cm]\draw (-.5,.-.25)--(.5,-.25)--(.5,.25)--(-.5,.25)--(-.5,-.25);\draw (.25,-.25)--(.25,-.5);\draw (0,-.25) arc (0 : -180 : .175);\node at (.55,-.4) {\tiny{2m-2}};\node at (0,0) { $x$ };\end{tikzpicture}=0}
\newcommand{\diagWtnbm}{
\begin{tikzpicture}[baseline=0cm]\draw (-.5,.-.25)--(.5,-.25)--(.5,.25)--(-.5,.25)--(-.5,-.25);\draw (-.25,.25)--(-.25,.5);\draw (0,-.25)--(0,-.5);\draw (.35,.25) arc (0 : 180 : .175);\node at (-.55,.4) {\tiny{2n-2}};\node at (0,0) { $x$ };\end{tikzpicture}=
\begin{tikzpicture}[baseline=0cm]\draw (-.5,.-.25)--(.5,-.25)--(.5,.25)--(-.5,.25)--(-.5,-.25);\draw (.25,.25)--(.25,.5);\draw (0,-.25)--(0,-.5);\draw (0,.25) arc (0 : 180 : .175);\node at (.55,.4) {\tiny{2n-2}};\node at (0,0) { $x$ };\end{tikzpicture}=
\begin{tikzpicture}[baseline=0cm]\draw (-.5,.-.25)--(.5,-.25)--(.5,.25)--(-.5,.25)--(-.5,-.25);\draw (.25,-.25)--(.25,-.5);\draw (0,.25)--(0,.5);\draw (0,-.25) arc (0 : -180 : .175);\node at (.55,-.4) {\tiny{2m-2}};\node at (0,0) { $x$ };\end{tikzpicture}=\begin{tikzpicture}[baseline=0cm]\draw (-.5,.-.25)--(.5,-.25)--(.5,.25)--(-.5,.25)--(-.5,-.25);\draw (-.25,-.25)--(-.25,-.5);\draw (0,.25)--(0,.5);\draw (.35,-.25) arc (0 : -180 : .175);\node at (-.55,-.4) {\tiny{2m-2}};\node at (0,0) { $x$ };\end{tikzpicture}=0
}
\newcommand{\diagdblecup}{\begin{tikzpicture}[baseline=0cm]\draw (-.2,0) arc (0 : 180 : .2);\draw (0,0) arc (0 : 180 : .4);\end{tikzpicture}}
\newcommand{\diagdblecap}{\begin{tikzpicture}[baseline=-.35cm]\draw (-.2,0) arc (0 : -180 : .2);\draw (0,0) arc (0 : -180 : .4);\end{tikzpicture}}

\newcommand{\diagmultxy}{xy = \sum_{a=0}^{ \min(2n,2i) } \sum_{ b=0 }^{ \min(2m,2j) }\begin{tikzpicture}[baseline = .425cm]
\draw (.25,.25)--(.75,.25)--(.75,.75)--(.25,.75)--(.25,.25);
\draw (1,.25)--(1.5,.25)--(1.5,.75)--(1,.75)--(1,.25);
\draw (0,.5)--(.25,.5);
\draw (.75,.5)--(1,.5);
\draw (1.5,.5)--(1.75,.5);
\draw (.375,0)--(.375,.25);
\draw (.375,.75)--(.375,1);
\draw (1.375,0)--(1.375,.25);
\draw (1.375,.75)--(1.375,1);
\draw (.625,.75) arc (180 : 0 : .25);
\draw (.625,.25) arc (-180 : 0 : .25);
\node at (.5,.5) { $x$ };
\node at (1.25,.5) { $y$ };
\node at (-.25,.5) {\tiny{2k} };
\node at (2,.5) {\tiny{2k} };
\node at (.875,1.125) {\tiny{a} };
\node at (.875,-.125) {\tiny{b} };
\end{tikzpicture}}

\begin{document}

\title{On fixed point planar algebras}
\maketitle
\begin{center}
{\sc by Arnaud Brothier\footnote{Department of Mathematics, University of Rome Tor Vergata, Via della Ricerca Scientifica, 1 - 00133 Roma, Italy, arnaud.brothier@gmail.com, https://sites.google.com/site/arnaudbrothier/}}
\end{center}

\begin{abstract}\noindent
To a weighted graph  can be associated a \BGPA\ $\Pl$.
We construct and study the \SEI\ of $\Pl.$
We show that this construction is equivariant with respect to the automorphism group of $\Pl.$
We consider subgroups $G$ of the automorphism of $\cP$ such that the $G$-fixed point space $\Pl^G$ is a \SPA.
As an application we show that if $G$ is amenable, then $\Pl^G$ is amenable as a \SPA.
We define the notions of a cocycle action of a \Hepa\ on a tracial \VNA\ and the corresponding \CrP.
We show that a large class of \SEI s of \SPA s can be described by such a \CrP.
\end{abstract}

\section*{Introduction and main results}

The theory of \SF s has been initiated by Jones \cite{Jones_index_for_subfactors}.
Given a \SF\ (an extremal unital inclusion of factors of \tII\ with finite index), Jones associated a combinatorial invariant called the \SI.
It has been axiomatized in the finite depth case as a paragroup by Ocneanu \cite{Ocneanu_quant_group_string_galois}.
Then, it has been axiomatized in the general case as a $\lambda$-lattice by Popa \cite{popa_system_construction_subfactor} and as a \SPA\ by Jones \cite{jones_planar_algebra}.
The reconstruction theorem of Popa shows that any $\lambda$-lattice is the \SI\ of a \SF, see \cite{Popa_Markov_tr_subfactors,popa_system_construction_subfactor,Popa_univ_constr_subfactors}.

Popa studied analytic properties of \SF s and among other defined the notion of amenability in this context \cite{Popa_classification_subfactors_amenable}. 
He gave many characterizations of amenability and defined it for $\lambda$-lattices and thus for \SPA s.
He constructed the so-called \SEI\ $\TS$ associated to a \SF\ $\NM$ \cite{Popa_94_Sym_env_alg} which extends constructions due to Ocneanu, and Longo and Rehren \cite{Ocneanu_quant_group_string_galois,Longo_Rehren_95_Nets_sf}.
The \SEI\ is a \SF\ of type II$_1$ and has finite index if and only if $\NM$ has finite depth.
Let $\Pl$ be the \SPA\ of a \SF\ $\NM$. 
Popa proved in \cite{Popa_symm_env_T} that $\Pl$ is amenable \IFF\ the \SEI\ of $\NM$ is co-amenable, see Section \ref{sec:amenability}.

Given a \SPA\ $\Ql$, one can construct a tower of II$_1$ factors $M_0\subset M_1\subset M_2\subset\cdots$ and a sequence of \SEI s $T_k\subset S_k, k\gO$ \cite{GJS_1,CJS}.
It has been proven that the \SI\ of $M_0\subset M_1$ is the \SPA\ $\Ql$ and the \SEI\ of $M_{k-1}\subset M_k$ is isomorphic to the inclusion $T_k\subset S_k$ for any $k\geqslant 1.$
We call $T_0\subset S_0$ the \SEI\ associated to $\Ql$.

Given a bipartite graph $\Ga$ and a weight $\mu$ on its edges satisfying certain assumptions we can construct a \PA\ $\Pl$ called a \BGPA\ \cite{Jones_planar_alg_bip_graph,Burstein_BGPA}. See also \cite{Morrison_Walker_embedding}.
The automorphism group $\Aut(\cP)$ of $\cP$ is isomorphic to a semi-direct product $U(\Pl_1^+)\rtimes \Aut(\Ga,\mu)$ where $\Aut(\Ga,\mu)$ is the group of automorphisms of $\Ga$ that preserves the weight $\mu$ and $U(\Pl_1^+)$ is the unitary group of the 1-box space $\Pl_1^+$.
The group $\Aut(\Pl)$ acts on the vertices $V=V^+\cup V^-$ of $\Ga$ where $U(\Pl_1^+)$ acts trivially.
If $G<\Aut(\cP)$ is any subgroup, then the \FPS\ $\Pl^G$ is a \PA. 
If $G$ acts transitively on the even and odd vertices of $\Ga$, then $\Pl^G$ satisfies automatically all the axioms of a \SPA\ except sphericality. 
In particular we obtain that $\Pl^G$ is non-degenerate which is usually the hardest axiom to check for a \SPA, see Section \ref{sec:preliminaries} for definitions.
Note that starting from a bipartite graph $\Ga$ and a vertex-transitive group action $G\act \Ga$, there exists a unique weight $\mu$ such that we can associate to $(\Ga,\mu)$ a \BGPA\ $\Pl$ and such that the \FPS\ $\Pl^G$ is a \SPA, see \cite[Proposition 2.5]{Arano_Vaes_SF}.

In this article, we consider \BGPA s $\Pl$ and their planar subalgebras $\Ql\subset\Pl$.
We investigate the structure of $\Ql$, its \SEI, and the notion of amenability when $\Ql$ is a \SPA\ embedded in $\Pl.$
We start by proving that $\Ql$ is non-amenable if $\Vert\Ga\Vert<\delta$, where $\delta$ is the modulus of $\Ql$.
The proof is elementary and is independent from the rest of the paper.
We extend the construction of \cite{GJS_1,CJS} and associate to a \BGPA\ $\Pl$ a tower of \VNA s $M_0\subset M_1\subset M_2\subset\cdots$ and a sequence of inclusions $T_k\subset S_k, k\gO$.
Note that the construction of the tower was already explained in \cite[Section 4]{GJS_1}.
We prove that $M_i$ and $T_k$ are \VNA s of type II$_1$ with atomic centers and that $S_k$ is a factor of type II for any $i,k\gO.$
We show that the automorphism group $\AutP$ of $\Pl$ acts in a minimal way on those \VNA s and show  that the constructions of $M_i$ and $S_k$ behaves as expected \WRT\ inclusions of planar algebras and group actions.
In particular, if $G<\AutP$ is a subgroup such that $\Pl^G$ is a \SPA, then the inclusion of fixed point spaces $M_0^G\subset M_1^G$ is a \SF\ with \SI\ isomorphic to $\Pl^G$ and the \SEI\ of this \SF\ (resp. of $\Pl^G$) is isomorphic to $M_1^G\vee (M_1^\op)^G\subset S_1^G$ (resp. $M_0^G\vee (M_0^\op)^G\subset S_0^G$).
As an application we prove the following theorem.

\begin{letterthm}\label{theo:introone}
Consider a \BGPA\ $\Pl$ over a weighted graph $(\Ga,\mu)$ and a subgroup $G<\AutP$ such that $\Pl^G$ is a \SPA.
If the group $G$ is amenable (as closed subgroup or a countable discrete group), then the \SPA\ $\Pl^G$ is amenable.
\end{letterthm}

Note that a more general framework has been studied independently by Arano and Vaes from which this theorem follows \cite{Arano_Vaes_SF}.
Recall that a \Hepa\ $(G,H)$ is an inclusion of groups $H<G$ which is \AN, i.e. for any $g\in G$ the group $H\cap gHg^{-1}$ has finite index inside $H$ and $gHg^{-1}.$
We define the notion of a cocycle action of a \Hepa\ on a tracial \VNA\ and the corresponding twisted \CrP\ \VNA.
Note that it has been considered already in the framework of ordinary action on $C^*$-algebras by Palma, see \cite{Palma_HP_I,Palma_HP_II}.
We show that if  $\mu$ is constant on the set of even edges and $G<\AutP$ is a subgroup such that $\Pl^G$ is a \SPA, then the \SEI\ of $\Pl^G$ can be described in the following way.

\begin{letterthm}\label{theo:introtwo}
Consider a \BGPA\ $\Pl$ over a weighted graph $(\Ga,\mu)$ such that $\mu$ is constant on the set of even edges and a subgroup $G<\AutP$.
Assume that $\Pl^G$ is a \SPA.
Denote by $G_o$ the subgroup of $G$ that fixes an even vertex $o$ of $\Ga.$
There exists a II$_1$ factor $A$ and a cocycle action of the \Hepa\ $(G,G_o)$ on $A\ootimes A^\op$ such that $G_o$ acts on $A$ and such that the \SEI\ of $\Pl^G$ is isomorphic to the inclusion $A^{G_o}\ootimes (A^\op)^{G_o}\subset \text{vN}[A\ootimes A^\op;G,G_o]$, where $\text{vN}[A\ootimes A^\op;G,G_o]$ is a  twisted \CrP\ of $A\ootimes A^\op$ by the \Hepa\ $(G,G_o).$ 
\end{letterthm}

See Theorem \ref{theo:SEIQ} for a slightly more precise statement.
Many \SEI\ of \SF s can be described in that way including diagonal and Bisch-Haagerup \SF s, see Section \ref{sec:examples}.
This theorem can be interpreted as an extension of a theorem of Popa which shows that the \SEA\ $S$ of a diagonal \SF\ is the twisted \CrP\ of a \VNA\ by a group \cite[Section 3]{Popa_classification_subfactors_amenable}.

\section{Preliminaries and a criterion of non-amenability}\label{sec:preliminaries}

A \PA\ is a collection of complex $*$-algebras $\Pl=(\Pl_n^\pm:n\gO)$ on which the operad of shaded planar tangles acts.
See \cite{jones_planar_algebra,jones_planar_algebra_II} for more details.
We follow similar conventions to \cite{CJS} for drawing a shaded planar tangle.
We decorate strings with natural numbers to indicate that they represent a given number of parallel strings.
The distinguished interval of a box is decorated by a dollar sign if it is not at the top left corner.
We do not draw the outside box and will omit unnecessary decorations.
The left and right traces of a \PA\ are the maps $\tau_l:\Pl_n^\pm\loriar\Pl_0^\mp$ and $\tau_r:\Pl_n^\pm\loriar\Pl_0^\pm$ defined for any $n\gO$ such that 
$$\tau_l(x)=\diagtaul \text{ and } \tau_r(x)=\diagtaur \text{ for any } x\in \Pl_n^\pm.$$
Suppose that $\Pl_0^\pm=\C.$
The \PA\ is called spherical if the two traces agree on each element of $\Pl.$
We say that $\Pl$ is non-degenerate if the sesquilinear forms $(x,y)\mapsto \tau_l(xy^*)$ and $(x,y)\mapsto \tau_r(xy^*)$ are positive definite.
A \SPA\ is a \PA\ such that each space $\Pl_n^\pm$ is finite dimensional, $\Pl_0^\pm=\C$, $\Pl$ is spherical, and is non-degenerate.
The modulus of a \SPA\ is the value of a closed loop.
The index of a \SPA\ is the square of its modulus.
Note that all the \SF s considered in this article are extremal. 
This condition corresponds to the sphericality of their associated \SPA.
`
\subsection{Bipartite graph planar algebras}

We refer to \cite{Jones_planar_alg_bip_graph} and \cite{Burstein_BGPA} for more details about \BGPA s of finite and infinite graphs respectively.
Let $\Ga$ be a countable locally finite undirected connected bipartite graph that can have multiple edges between two vertices.
Denote by $V=V^+\sqcup V^-$ its set of vertices, where $V^+$ and $V^-$ are the set of even and odd vertices respectively.
We consider the associated symmetric oriented graph obtained by doubling each edge of $\Ga$ into a pair of oppositely oriented edges.
If $a$ is a path, then we denote by $\bar a, s(a), t(a)$ its associated opposite edge, its source, and its target respectively.

We put $C_0^\pm=V^\pm$, $C_n^\pm,n\geqslant 1$ the set of paths in $\Ga$ of length $n$ that start in $V^\pm$, $C_*^\pm=\bigcup_{n\geqslant 0} C_n^\pm.$
Consider the \HS\ $\ell^2(C_n^\pm)$ with \ONB\ indexed by $C_n^\pm$.
Let $B(\ell^2(C_n^\pm))$ be the space of bounded linear operators on $\ell^2(C_n^\pm)$ with the standard \SMU\ $\{e_{a,b}: a,b\in C_n^\pm\}$.
Define the von Neumann subalgebra $\Pl^\pm_n\subset B(\ell^2(C_n^\pm))$ generated by the operators $e_{a,b}$ where $a,b$ are in $C_n^\pm$ and have the same source and target.
We denote by $ST_n^\pm,n\geqslant 0$ the set of couples $(a,b)$ where $a,b\in C_n^\pm,s(a)=s(b),$ and $t(a)=t(b)$.
We put $ST_*^\pm=\cup_{n\geqslant 0} ST_n^\pm.$
Observe that the \VNA\ $\Pl_0^\pm$ is isomorphic to the abelian atomic \VNA\ of bounded functions $\ell^\infty(V^\pm)$. 
We identify $\Pl_0^\pm$ with $\ell^\infty(V^\pm)$.
We denote by $\{e_v:v\in V^\pm\}$ the set of minimal projections of $\Pl_0^\pm.$

Assume there exists a strictly positive map $\mu:C_1\loriar \R_+^*$ called a weight satisfying that for any paths $a=a_1\cdots a_n, b=b_1\cdots b_l $ of length $n,l\geqslant 1$ such that $s(a_1)=s(b_1), t(a_n)=t(b_l)$ we have that
\begin{equation}\label{equa:mu_inversion}
\mu(a_1)\cdots\mu(a_n)=\mu(b_1)\cdots\mu(b_l) \text{ and}
\end{equation}
\begin{equation}\label{equa:mu_delta}
\text{ there exists } \delta>0 \text{ such that } \sum_{c\in C_1: s(c)=v} \mu(c)=\delta \text{ for any } v\in V.
\end{equation}
Such a pair $(\Ga,\mu)$ is called a weighted graph with modulus $\delta$ or simply a weighted graph.
Note that a weaker notion has been introduced and studied in \cite{DeCommer_Yamashita_duality_II} where the first assumption is replaced by $\mu(a)\mu(\bar a)=1$ for any edge $a$ where $\bar a$ denotes the opposite edge of $a$.

Recall that the degree deg$(v)$ of a vertex $v$ is the number of edges with source $v$ and the degree deg$(\Ga)$ of the graph $\Ga$ is define as $\sup_{v\in V} \text{deg}(v).$
The existence of the weight $\mu$ implies that deg$(\Ga)\leqslant \delta^2.$
Indeed, by \eqref{equa:mu_delta}, we have that $\mu(a)\leqslant \delta$ for any $a\in C_1.$
The equality \eqref{equa:mu_inversion} and the inequality of above implies that $\delta^{-1}\leqslant \mu(a)\leqslant\delta$ for any $a\in C_1$.
By using \eqref{equa:mu_delta} again, we obtain that $\delta=\sum_{a\in C_1:s(a)=v}\mu(a)\geqslant \text{deg}(v) \delta^{-1}$ for any $v\in V.$
Therefore, deg$(\Ga)\leqslant\delta^2.$

We extend $\mu$ on the set of all paths of $\Ga$ by putting $\mu(a_1\cdots a_n)=\mu(a_1)\cdots \mu(a_n)$ where $a_1\cdots a_n$ is the concatenation path of $n$ edges $a_1,\cdots,a_n.$
Up to multiplication by a strictly positive real number, there exists a unique map $\mu_V:V\loriar\R_+^*$ such that $\mu(a)=\mu_V(t(a))/\mu_V(s(a))$ for any $a\in C_1$.
Indeed, let us fix a vertex $o\in V$ and put $\mu_V(o)=1$ and $\mu_V(v)=\mu(a)$ where $a$ is any path such that $s(a)=o$ and $t(a)=v$.
Since $\Ga$ is assumed to be connected, there is always a path from $o$ to $v$.
Condition \eqref{equa:mu_inversion} assures that $\mu_V$ is well defined.
Suppose there is another weight $\nu$ satisfying that $\mu(a)=\nu(t(a))/\nu(s(a))$ for any $a\in C_1$.  We can see that $\nu(v)=\mu_V(v)\nu(o)$ for any $v\in V.$ Hence, $\mu_V$ is unique up to a multiplication by a strictly positive real number.
This map satisfies that $A_\Ga(\mu_V)=\delta\mu_V$, where $A_\Ga$ is the adjacency matrix of the graph $\Ga.$
Following \cite{Jones_planar_alg_bip_graph,Burstein_BGPA}, the data of $(\Ga,\mu_V)$ allows us to define a \PA\ structure on $\Pl_\Ga=\Pl=(\Pl_n^\pm : n\gO)$.
Note that this \PA\ structure only depends on $\Ga$ and $\mu$.
We say that $\Pl$ is the \BGPA\ associated to the weighted graph $(\Ga,\mu).$

The \BGPA\ $\Pl$ is known to satisfy the identity
\begin{equation}\label{equa:traces}
\tau_l(e_{a,b})=\delta_{a,b}e_{t(a)}\mu(\bar a) \text{ and } \tau_r(e_{a,b})=\delta_{a,b}e_{s(a)}\mu(a)
\end{equation}
 for any $(a,b)\in ST_*^\pm$, where $\delta_{a,b}$ is the Kronecker symbol.

Note that condition \eqref{equa:mu_inversion} on $\mu$ assures the existence of $\mu_V$ and implies that $\tau_l$ and $\tau_r$ are tracial.
More precisely, consider the collection of $*$-algebras $(\Pl_n^\pm:n\geqslant 0)$ and the linear functional $\tau_l,\tau_r:\Pl_n^\pm\loriar\Pl_0,n\geqslant 0$ defined on the \SMU\ by equation \eqref{equa:traces}.
Then $\tau_l,\tau_r$ are tracial \IFF\ $\mu(a)=\mu(b)$ for any $(a,b)\in ST_*^\pm.$
Indeed, suppose that $\tau_r$ is tracial.
If $(a,b)\in ST_*^\pm$, then $\tau_r(e_{a,b} e_{b,a})=e_{s(a)}\mu(a)=\tau_r(e_{b,a}e_{a,b})=e_{s(a)} \mu(b).$
Conversely, suppose that $\mu(a)=\mu(b)$ for any $(a,b)\in ST_*^\pm.$
Consider $n\geqslant 1, x=\sum_{(a,b)\in ST_n^\pm} x_{a,b} e_{a,b}, y=\sum_{(a,b)\in ST_n^\pm} y_{a,b} e_{a,b}.$
We have that $\tau_r(xy)=\sum_{(a,b)\in ST_n^\pm} x_{a,b} y_{b,a} \mu(a) e_{s(a)}$ and $\tau_r(yx)=\sum_{(a,b)\in ST_n^\pm} y_{b,a} x_{a,b} \mu(b) e_{s(b)}=\sum_{(a,b)\in ST_n^\pm} x_{a,b} y_{b,a} \mu(a) e_{s(a)}=\tau_r(xy).$
A similar proof shows that $\tau_l$ is also tracial.

\subsection{Amenability}\label{sec:amenability}
All groups that we consider are either countable discrete or locally compact second countable.
This implies that any quotient space by a closed subgroup is countable.
Recall that a locally compact group $G$ is amenable if and only if for any affine action of $G$ on a compact convex subset of a locally convex topological vector space there exists a $G$-fixed point.
Eymard defined and studied amenability of inclusion of groups \cite{Eymard_Coset_amen}.
A closed subgroup $H<G$ is co-amenable if and only if for any affine action of $G$ on a compact convex subset of a locally convex topological vector space that admits a $H$-fixed point there exists a $G$-fixed point.
Remark that if $G$ is an amenable group, then any closed subgroup $H<G$ is co-amenable.
There are alternative definitions of co-amenability when the subgroup $H<G$ is \AN, i.e. for any $g\in G$ the group $H\cap gHg^{-1}$ has finite index inside $H$ and $gHg^{-1}.$
It has been proven in \cite[Theorem 3.8]{Delaroche_AP_Coset_spaces} that a closed \AN\ subgroup $H<G$ is co-amenable if and only if there exists a sequence of positive definite $H$-bi-invariant maps $f_n:G\loriar\C, n\gO$ with support contained in a finite union of right $H$-cosets and that converges pointwise to $1$. We will use later this characterization.

Recall that if $\mN\subset \mM$ is an inclusion of II$_1$ factors and $L^2(\mathcal M,\tau)$ is the GNS \HS\ associated to the unique \nofa\ tracial state of $\mM$ we say that a $\mathcal N$-bimodule $K\subset L^2(\mathcal M,\tau)$ is bifinite if there exists a finite subset $F\subset L^2(\mM,\tau)$ such that $K$ is contained in the closure of Span$(x\cdot\xi ; x\in\mN,\xi\in F)$ and in the closure of Span$(\xi\cdot x : x\in \mN,\xi\in F).$
Similarly to the group case, we say that an inclusion of II$_1$ factors $\mN\subset\mM$ is co-amenable if there exists a sequence of \copo\ $\mN$-bimodular maps $\phi_n:\mM\loriar\mM$ such that their image is a bifinite $\mN$-bimodule and $\lim_{n\rightarrow\infty}\Vert\phi_n(x)-x\Vert_2=0$ for any $x\in \mM.$

Let $\Ql$ be a \SPA\ with loop parameter $\delta$.
Popa defined amenability for \SF s and $\lambda$-lattices \cite{Popa_classification_subfactors_amenable}.
In planar algebraic terms, $\Ql$ is amenable if $\Vert \Ga(\Ql)\Vert=\delta$, where $\Ga(\Ql)$ is the principal graph of $\Ql$ deduced from the Bratteli diagram of the tower of finite dimensional semi-simple $*$-algebras $\Ql_0^+\subset \Ql_1^+\subset\cdots$ and where $\Vert\Ga(\Ql)\Vert$ is the operator norm of the adjacency matrix of $\Ga$ acting on the $\ell^2$-space of the vertices of $\Ga(\Ql).$
If $\Ql$ is a \SPA, then we can associate a sequence of \SEI s $T_k\subset S_k, k\gO$ as constructed in \cite{CJS}.
Those are inclusions of II$_1$ factors. It can be shown that $T_1\subset S_1$ is isomorphic to the \SEI\ of a subfactor having its \SPA\ isomorphic to $\Ql$, see \cite[Theorem 3.3]{CJS}.
Popa proved that $\Ql$ is amenable if and only if the \SEI\ $T_1\subset S_1$ is co-amenable \cite[Theorem 5.3]{Popa_symm_env_T}.
Furthermore, up to a compression and finite index inclusions we have that $T_1\subset S_1$ is isomorphic to $T_0\subset S_0$ \cite[Lemma 4.3]{Brot_Jones_AP_SPA}.
Therefore, $\Ql$ is amenable if and only if $T_0\subset S_0$ is co-amenable.
We will later use this characterization of amenability of \SPA s.

\subsection{A criterion of amenability for \SPA s}\label{sec:amenability}

We provide a criteria of non-amenability for a \SPA\ embedded in a \BGPA.
The proof of this criteria is elementary and is interesting by its own.

\begin{theo}\label{theo:amenable}
Let $(\Ga,\mu,\delta)$ be a weighted graph with a modulus such that $\Vert \Ga\Vert<\delta.$
Consider its associated \BGPA\ $\Pl.$
If $\Ql$ is a \SPA\ that embeds inside $\Pl$, then $\Ql$ is non-amenable.
\end{theo}

\begin{proof}
Let $\Ga,\mu,\delta,\Ql$ be as above.
Denote by $\Ga(\Ql)$ the principal graph of $\Ql.$
It is sufficient to show that $\Vert\Ga(\Ql)\Vert\leqslant \Vert \Ga\Vert$ since $\Ql$ and $\Pl$ have the same modulus $\delta.$
Let $\Ga(\Ql)_n$ be the Bratteli diagram of the inclusion $\Ql^+_n\subset \Ql_{n+1}^+$ for $n\gO.$
We have the equality 
\begin{equation}\label{equ:am1}
\Vert\Ga(\Ql)\Vert=\lim_{n\rightarrow\infty}\Vert\Ga(\Ql)_n\Vert.
\end{equation}
Let $\Delta_n$ be the Bratteli diagram of the inclusion $\Pl^+_n\subset \Pl_{n+1}^+$.
This graph is equal to the disjoint union of graphs $\Delta_n=\bigsqcup_{v\in V^+} \Delta_n(v)$, where $\Delta_n(v)$ is the subgraph of $\Ga$ with vertices $V_n(v)^\pm=\{ w \in V^\pm: d(v,w)\leqslant n+1 \}$ and edges the one of $\Ga$.
Since $\Delta_n$ is a subgraph of $\Ga$ we have that
\begin{equation}\label{equ:am2}
\Vert \Delta_n\Vert\leqslant \Vert \Ga\Vert, \text{ for any }n\geqslant 0.
\end{equation}
The \PA s $\Ql$ is contained inside $\Pl$.
Fix $n\gO$ and consider the following square of inclusions
\begin{equation}\label{equa:square}
\begin{array}{ccc}
\Pl_{n}^+ &\subset & \Pl_{n+1}^+ \\
\cup & & \cup \\
\Ql_{n}^+ & \subset & \Ql_{n+1}^+
\end{array}.\end{equation}
Consider the normalized tracial operator 
$$tr_m:\Pl_{m}^+\loriar\Pl_0^+,x\longmapsto \frac{1}{\delta^m}\diagtaur \text{ for any } m\gO.$$
Observe that the restriction of $tr_{n+1}$ to $\Pl_n^+$ is equal to $tr_n$.
We equipped $\Pl_n^+,\Ql_n^+$, and $\Ql_{n+1}^+$ with the corresponding restrictions of $tr_{n+1}.$
Note that $$E_\Pl: \Pl_{n+1}^+\loriar \Pl_n^+, x\longmapsto \diagEP$$
is a faithful \CE\ satisfying $tr_{n}\circ E_\Pl=tr_{n+1}.$
The restriction of $E_\Pl$ to $Q_{n+1}^+$ is a trace-preserving \CE\ from $\Ql_{n+1}^+$ onto $\Ql_n^+$.
Therefore, the square of inclusion \eqref{equa:square} is a commuting square of inclusions with respect to the \CE s $E_\Pl$ and $E_\Pl\vert_{\Ql_{n+1}^+}.$
Let $L_n$ be the \VNA\ generated by $\Pl_n^+$ and $\Ql_{n+1}^+$ inside $\Pl_{n+1}^+.$
Denote by $\Lambda_n$ the Bratteli diagram of the inclusion $\Pl_n^+\subset L_n$.
The inclusions $\Ql_n^+\subset \Ql_{n+1}^+\subset L_n$ and $\Pl_n^+\subset L_n$ form a non-degenerate commuting square.
Therefore, $\Vert\Ga(\Ql)_n\Vert= \Vert \Lambda_n\Vert$ by \cite[Remark 5.3.5]{Jones_Sunder_subfactors}.
This implies that 
\begin{equation}\label{equ:am3}
\Vert \Ga(\Ql)_n\Vert\leqslant \Vert \Delta_n\Vert, \text{ for any } n\gO.
\end{equation}
We obtain the result by combining \eqref{equ:am1}, \eqref{equ:am2}, and \eqref{equ:am3}.
\end{proof}

\begin{remark}
The embedding theorem of \cite{Jones_Penneys_embedding_pla_alg,Morrison_Walker_embedding} implies the following.
If $\Ql$ is a non-amenable \SPA, then there exists a \BGPA\ $\Pl$ associated to a weighted graph with a modulus $(\Ga,\mu,\delta)$ such that $\Vert\Ga\Vert<\delta$ and such that $\Ql$ embeds inside $\Pl.$
Therefore, a \SPA\ $\Ql$ is non-amenable if and only if there exists a weighted graph with a modulus $(\Ga,\mu,\delta)$ such that $\Vert\Ga\Vert<\delta$ and such that $\Ql$ embeds in the \BGPA\ associated to $(\Ga,\mu,\delta)$.
\end{remark}

\section{Construction of \VNA s}

\subsection{Von Neumann algebras associated to a weighted graph with a modulus}\label{sec:construction}

We fix a weighted graph $(\Ga,\mu)$ with modulus $\delta$ and consider the \BGPA\ $\Pl=\Pl_\Ga$.
We generalize the construction of \cite{GJS_1,JSW} and \cite{CJS} and provide a tower of \VNA s and a family of \SEI s associated to $\Pl$.
Let $k\geqslant 0$ be a natural number and $\epsilon$ the sign $+$ is $k$ is even and $-$ if $k$ is odd.
For any $n,m\gO,$ let $D_k(n,m)$ be a copy of the vector space $\Pl^+_{n+m+2k}$.

Consider the direct sum
$$Gr_k\Pl\boxtimes Gr_k\Pl:=\bigoplus_{n,m\gO}D_k(n,m)$$
that we equipped with the Bacher product:
$$\diagmultxy$$
where $x\in D_k(n,m), y\in D_k(i,j),$ and there are $2k$ horizontal strands in the middle.

Let $\dag:Gr_k\Pl\boxtimes Gr_k\Pl\loriar Gr_k\Pl\boxtimes Gr_k\Pl$ be the anti-linear involution that sends $D_k(n,m)$ to itself and satisfies
$$x^\dag=\diagxdag\ , \text{ for any } x\in D_k(n,m).$$
Consider the linear map $E:Gr_k\Pl\boxtimes Gr_k\Pl\loriar\Pl_0^\epsilon$ which sends $x\in D_k(0,0)$ to 
$$\delta^{-2k}\diagthree$$
and $0$ to any element in $D_k(n,m)$ with $(n,m)\neq (0,0).$
The vector space $Gr_k\Pl\boxtimes Gr_k\Pl$ endowed with those operation is an associative $*$-algebra with a faithful positive linear map to $\Pl_0^\epsilon$.
We fix a weight $\mu_V:V\loriar\R_+^*$ that satisfies that $\mu(a)=\mu_V(t(a))/\mu_V(s(a))$ for any $a\in C_1.$
Consider the \nofa\ semi-finite weight $\tau_V$ satisfying $\tau_V(e_v)=\mu_V(v)^2$ for any $v\in V.$
Denote by $H_k(n,m)$ the \HS\ equal to the completion of the pre-\HS\ $\{x\in D_k(n,m): \tau_V\circ E(xx^\dag)<\infty\}$ equipped with the inner product $\langle x,y\rangle=\tau_V\circ E(xy^\dag)$, for any $n,m\gO$.
Let $H_k=\bigoplus_{n,m\gO} H_k(n,m)$ be the direct sum of those \HS s.
Consider $\pi_k:Gr_k\Pl\boxtimes Gr_k\Pl\loriar\mathcal L(K_k)$, the left regular representation of $Gr_k\Pl\boxtimes Gr_k\Pl$, where $K_k=H_k\cap Gr_k\Pl\boxtimes Gr_k\Pl$ and  $\mathcal L(K_k)$ is the algebra of linear maps from $K_k$ to $K_k$.

\begin{prop}
For any $x\in Gr_k\Pl\boxtimes Gr_k\Pl$, the linear map $\pi_k(x)$ defines a bounded operator on $H_k$.
\end{prop}
\begin{proof}
The proof follows the same ideas as \cite[Theorem 3.3]{JSW}.
\end{proof}

Denote by $S_k$ the \VNA\ generated by the image of $Gr_k\Pl\boxtimes Gr_k\Pl$ inside $B(H_k)$.
We continue to denote by $\pi_k$ the representation of $S_k$ on $H_k.$
Let $Gr_k\Pl$ be the subalgebra of $Gr_k\Pl\boxtimes Gr_k\Pl$ generated by $\bigcup_{n\gO}D_k(n,0)$.
We identify the opposite algebra of $Gr_k\Pl$ with the subalgebra of $Gr_k\Pl\boxtimes Gr_k\Pl$ generated by $\bigcup_{m\gO} D_k(0,m)$ that we denote by $Gr_k\Pl^\op.$
Let $\C[V^\epsilon]\subset \ell^\infty(V^\epsilon)$ be the subalgebra generated by the set of projections $\{e_v, v\in V^\epsilon\}$.
The abelian algebra $\C[V^\epsilon]$ is contained in the center of $Gr_k\Pl$ and $Gr_k\Pl^\op$.
Further, the subalgebra of $Gr_k\Pl\boxtimes Gr_k\Pl$ generated by $Gr_k\Pl$ and $Gr_k\Pl^\op$ is isomorphic to the tensor product of $Gr_k\Pl$ and $Gr_k\Pl^\op$ over $\C[V^\epsilon]$.
We denote this subalgebra by $$Gr_k\Pl\otimes_{\C[V^\epsilon]}Gr_k\Pl^\op.$$ 

Denote by $T_k$ and $M_k$ the von Neumann subalgebras of $S_k$ generated by $Gr_k\Pl\otimes_{\C[V^\epsilon]}Gr_k\Pl^\op$ and $Gr_k\Pl$ respectively.
The \VNA\ generated by $Gr_k\Pl^\op$ is isomorphic to the opposite \VNA\ of $M_k$. We denote it by $M_k^\op.$
There exists an inclusion of $*$-algebras $i_k:Gr_k\Pl\otimes_{\C[V^\epsilon]}Gr_k\Pl^\op\loriar Gr_{k+1}\Pl\otimes_{\C[V^\varepsilon]}Gr_{k+1}\Pl^\op$ which consists of adding two horizontal lines in the middle and dividing by $\delta^2,$ where $\varepsilon$ is $+$ if $k$ is odd and $-$ if $k$ is even.
This inclusion induces unital embeddings from $T_k$ into $T_{k+1}$, $M_k$ into $M_{k+1}$, and $M_k^\op$ into $M_{k+1}^\op.$

\subsection{Properties associated to those von Neumann algebras}

For any $n,m\gO$, denote by $j_k(n,m)$ the canonical embedding of $\Pl^+_{n+m+2k}$ into $D_k(n,m)$ viewed as a subspace of $S_k.$
For any vertex $v\in V^\epsilon$, consider the projection $$\diagpkv$$
This projection belongs to $M_k\cap M_k^\op$ and commutes with $T_k$.
Denote by $T_k(v), M_k(v), M_k(v)^\op$ the corners $T_k p_v, M_kp_v,$ and $M_k^\op p_v$ respectively.

\begin{prop}\label{prop:vna}
\TFAT
\begin{enumerate}
\item The \VNA\ $T_k(v)$ is isomorphic to the tensor product $M_k(v)\ootimes M_k(v)^\op$ and $M_k(v)$ is a II$_1$ factor, for any $v\in V^\epsilon.$
\item The corner $p_vS_kp_v$ is equal to $T_k(v)$ for any $v\in V^\epsilon.$
\item The relative commutant $T_k'\cap S_k$ is equal to the center $Z(T_k)$ of $T_k$ and is isomorphic to $\Pl_0^\epsilon.$
The set of minimal central projections of $T_k$ is $\{p_v: v\in  V^\epsilon\}$.
\item The \VNA\ $S_k$ is a type II factor. It is a finite factor if and only if the graph $\Ga$ is finite.
Up to a scalar, the unique \nofa\ semi-finite tracial weight of $S_k$ is defined as $Tr(x)=\sum_{v\in \Ve} \langle \pi_k(x) p_v,p_v\rangle$ for any positive operator $x\in S_k.$
Moreover, $Tr(y)=\tau_V\circ E(y)$ for any positive operator $y\in Gr_k\Pl\boxtimes Gr_k\Pl.$   
\end{enumerate}
\end{prop}

\begin{proof}
We assume that $k=0$ and drop the subscript $k$.
The general case can easily be deduced.

Proof of (1).
Consider a vertex $v\in V^+$. 
It is obvious that $M(v)$ and $M(v)^\op$ commute and generate the \VNA\ $T(v)$.
Following \cite[Corollary 4.12]{JSW}, we obtain that $M(v)$ is a II$_1$ factor.
Let $\tau$ be the \nofa\ tracial state of $M(v).$
Consider the map $E: Gr\Pl\boxtimes Gr\Pl\loriar\Pl_0^+$ defined in Section \ref{sec:construction} and its restriction $\varphi$ to $T(v)\cap Gr\Pl\boxtimes Gr\Pl$ that has values in $\C p_v\simeq \C.$
Observe that $\varphi(a_1b_1^\op\cdots a_n b_n^\op)=\tau(a_1\cdots a_n)\tau (b_1\cdots b_n)$ for any $a_1,\cdots, a_n\in p_vGr\Pl$ and $b_1^\op,\cdots,b_n^\op\in p_vGr\Pl^\op.$
This implies that $T(v)$ is isomorphic to $M(v)\ootimes M(v)^\op.$

Proof of (2).
Consider a loop $l$ of length $2n+2m$ that starts at an even vertex $v.$
It defines a partial isometry $e_{a,b}\in \Pl^+_{n+m}\subset  B(C_{n+m}^+)$, where $a$ is the path equal to the first half of the loop and $b$ the opposite path equal to the second half of the loop.
Consider the corresponding element $x_l:=j(n,m)(e_{a,b})\in S.$
If $p_vx_lp_v=x_l$, then the 2n-th vertex in the loop $l$ is also equal to $v$.
This implies that $l$ is the concatenation of two loops $l_1$ and $l_2$ where $l_1$ is the truncation of $l$ for the 2n-th first edges and $l_2$ is the rest of the loop $l$.
Let $y_i=e_{a_i,b_i}$ be the partial isometry of $\Pl_n^+$ if $i=1$ and $\Pl_m^+$ if $i=2$ such that the concatenation of the path $a_i$ and the opposite of $b_i$ is equal to the loop $l_i$ for $i=1,2$.
We have that $x_l=x_{l_1}\otimes x_{l_2}$, where $x_{l_1}=j(n,0)(e_{a_1,b_1})$ and $x_{l_2}=j(0,m)(e_{a_2,b_2})$, 
$$\text{i.e. } \diagxl$$
Therefore, $p_vD(n,m)p_v$ is contained in $T(v)$ for any $v\in V^+$ and $n,m\gO.$
By density, we obtain that $p_vSp_v=T(v).$

Proof of (3).
Consider $x$ in the relative commutant $T'\cap S$.
For any $v,w\in V^+$ we have that $p_vxp_w=p_vp_wx.$
Therefore, $x$ is in $T$ by (2).
We obtain that $x$ is in the center of $T$.
We conclude by using (1).

Proof of (4).
Consider an element $x$ in the center of $S$.
By (3) this element belongs to the center of $T$.
Hence, there exists a bounded map $f:V^+\loriar\C$ such that $x=\sum_{v\in V^+} f(v)p_v$.
Consider two vertices $v$ and $w$.
There exists a path $c$ of length $n$ in $\Ga$ that starts at $v$ and ends at $w$ since the graph is connected.
Consider the corresponding non-zero projection $e_{c,c}\in \Pl_n^+$ and its image $y:=j(n,n)(e_{c,c})$ in $S$.
Observe that $yx=f(w)y$ and $xy=f(v)y$. Therefore, $f$ is constant and the center of $S$ is trivial.
We deduce that $S$ is a factor.
We have that $p_vSp_v$ is a II$_1$ factor by (1) and (2).
Therefore, $S$ is a type II factor.
Consider the weight $Tr$ of $S$ defined by $Tr(x)=\sum_{v\in V^+}\langle \pi(x)p_v,p_v\rangle,$ where $\langle\cdot,\cdot\rangle$ is the inner product of $H.$
It is clear that this weight is normal, semi-finite, faithful, and sends $p_v$ to $\mu_V(v)^2$ for any $v\in V^+.$
Consider a positive operator $y\in Gr\Pl\boxtimes Gr\Pl.$
Observe that
\begin{align*}
\sum_{v\in V^+} \langle  \pi(y) p_v , p_v \rangle & = \sum_{v\in V^+} \tau_V\circ E(y p_v p_v^*) = \sum_{v\in V^+} \tau_V\circ E(yp_v)\\
& = \tau_V\circ E(y \sum_{v\in V^+} p_v) = \tau_V\circ E(y).
\end{align*}

Consider two loops $l,k$ of length $2n+2m$ and $2t+2s$ in $\Ga$ that start at $v\in V^+$ and $w\in V^+$ respectively.
Consider $x_l$ and $x_k$ the corresponding elements of $S$ given by the inclusion maps $j(n,m)$ and $j(t,s).$
The identity \eqref{equa:traces} of Section \ref{sec:preliminaries} implies that $Tr(x_lx_k)=\delta_{n,t}\delta_{m,s}\mu_V(w)\mu_V(v)=Tr(x_kx_l),$ where $\delta_{n,t}$ is the Kronecker symbol.
We obtain by density that $Tr$ is tracial.
By uniqueness of a \nofa\ tracial weight on a type II factor, we have that $S$ is a finite \VNA\ \IFF\ $Tr(1)<\infty.$
Suppose that $\Gamma$ is finite.
Then $Tr(1)=\sum_{v\in V^+} Tr(p_v)=\sup_{v\in V^+}Tr(p_v) \vert V^+\vert<\infty,$ where $\vert V^+\vert$ is the cardinal of $V^+.$
Hence, $S$ is a II$_1$ factor.
Suppose that $\Gamma$ is infinite.
Assume that $\{\mu_V(v): v\in V^+\}$ is unbounded.
Then, $Tr(1)\geqslant \sup_{v\in V^+} \mu_V(v)^2=\infty.$
Assume that $\{\mu_V(v): v\in V^+\}$ is bounded by $C>0.$
Since $\delta^{-1}\leqslant\mu(a)\leqslant \delta$ and $\mu(a)=\mu_V(t(a))/\mu_V(s(a))$ for any $a\in C_1$, we obtain that $\{\mu_V(v): v\in V^+\}$ is bounded below by a constant $D>0.$
Then, $Tr(1)=\sum_{v\in V^+} Tr(p_v)=\sum_{v\in V^+} \mu_V(v)^2\geqslant D \vert V^+\vert=\infty$.
Therefore, $S$ is a II$_\infty$ factor if $\Gamma$ is infinite.
\end{proof}
We denote by $L^2(S_k,Tr)$ the GNS \HS\ associated to $S_k$ and $Tr$ and identify it with the \HS\ $H_k$.

\subsection{Planar algebras contained in a \BGPA}\label{sec:subalgebras}

Consider a \PA\ $\Ql$ which embeds in the \BGPA\ $\Pl.$
We identify $\Ql$ and its image in $\Pl.$
Consider the $*$-algebras $Gr_k\Ql, Gr_k\Ql^\op,$ and $Gr_k\Ql\boxtimes Gr_k\Ql$ that we identify with subalgebras of $Gr_k\Pl, Gr_k\Pl^\op,$ and $Gr_k\Pl\boxtimes Gr_k\Pl$ respectively.
Denote by $S_k(\Ql)$, $M_k(\Ql)$, and $M_k(\Ql)^\op$ the von Neumann subalgebras of $S_k$ generated by $Gr_k\Ql\boxtimes Gr_k\Ql$,  $Gr_k\Ql$, and $Gr_k\Ql^\op$ respectively.
Let $T_k(\Ql)=M_k(\Ql)\vee M_k(\Ql)^\op$ be the von Neumann subalgebra of $S_k$ generated by $M_k(\Ql)$ and $M_k(\Ql)^\op.$

\begin{prop}\label{prop:irreducible_subfactor}
Let $\J$ be the \TLJ\ \PA\ included in $\Pl$, i.e. the planar subalgebra of $\Pl$ generated by tangles without inner discs.
Then, $S_k(\J)\subset S_k$ is an irreducible subfactor.
Moreover, $T_k(\J)'\cap S_k= Z(T_k)$.
\end{prop}

\begin{proof}
We assume that $k=0$ and drop the subscript $k$.
The general case can easily be deduced.
Consider the element $\cup\in D(1,0)$ and $\cap\in D(0,1).$
Denote by $C$ and $C^\op$ the abelian \VNA s generated by $\cup$ and $\cap$ respectively.
Let $A=C\vee C^\op$ be the abelian \VNA\ generated by $C$ and $C^\op.$
Consider the \VNA\ $B=A\vee Z(T)$ generated by $A$ and the set of projections $\{p_v:v\in V^+\}$.
We claim that the relative commutant $S\cap A'$ is contained in $B$.
Consider the following subspaces of $S$:
$$W^{t,n}(v)= \{x\in D(n,0) p_v: \diagWtn\},$$
$$W_{b,m}(v)= \{x\in D(0,m) p_v: \diagWbm\}, \text{ and }$$
$W^{t,n}_{b,m}(v,w)$ the space of $x\in p_wD(n,m) p_v$ such that $$\diagWtnbm, \text { for any } n,m\geqslant 1, v,w\in V^+.$$
Observe that all those spaces are contained in $L^2(S,Tr)$ and are pairwise orthogonal.
Denote by $W^t(v)=\bigoplus_{n\geqslant 1} W^{t,n}(v)$, $W_b(v)=\bigoplus_{m\geqslant 1} W_{b,m}(v)$, and $W^t_b(v,w)=\bigoplus_{n,m\geqslant 1} W^{t,n}_{b,m}(v,w)$ the direct sum of those spaces inside $L^2(S,Tr)$ for any $v,w\in V^+.$
Let $H^t(v), H_b(v)$, and $H^t_b(v,w)$ be the $A$-bimodules generated by $W^t(v), W_b(v)$, and $W^t_b(v,w)$ respectively for any $v,w \in V^+.$
Consider the $A$-bimodule generated by $D(0,0)$ inside $L^2(S,Tr).$
It is isomorphic to $L^2(A)\otimes \ell^2(V^+)$ as a $A$-bimodule. 
We identify those two bimodules.
A similar proof to \cite[Theorem 4.9]{JSW} shows that we have the following decomposition into $A$-bimodules:
\begin{equation}\label{equa:decompo}
L^2(S,Tr)=(L^2(A)\otimes \ell^2(V^+))\oplus\bigoplus_{v\in V^+} H^t(v)\oplus\bigoplus_{v\in V^+} H_b(v)\oplus\bigoplus_{v,w\in V^+} H_b^t(v,w)
\end{equation}
Consider $x\in S\cap A'$ and two different even vertices $v\neq w$.
Observe that $p_v$ and $p_w$ commute with $A$.
Furthermore, $p_vxp_w$ is in $L^2(S,Tr).$
Hence, $p_vxp_w$ is in $L^2(S,Tr)\cap A'.$
The equality \eqref{equa:decompo} implies that $p_vxp_w$ is in $H^t_b(v,w).$
Therefore, $p_vxp_w$ is a $A$-central vector of $H^t_b(v,w).$
Following \cite[Theorem 4.9]{JSW}, we can prove that the $A$-bimodule $H^t_b(v,w)$ is isomorphic to $L^2(A)\otimes W_b^t(v,w)\otimes L^2(A)$.
Therefore, it is isomorphic to a direct sum of the coarse $A$-bimodule $L^2(A)\otimes L^2(A).$
Thus, $H^t_b(v,w)$ does not admit any nonzero $A$-central vectors.
Therefore, $p_vxp_w=0$ for any $v\neq w.$
The element $p_vxp_v$ is a $A$-central vector of $L^2(S,Tr).$
The equality \eqref{equa:decompo} implies that $p_vxp_v\in L^2(A p_v)\oplus H^t(v)\oplus H_b(v)\oplus H^t_b(v,v).$
By the previous argument we have that the orthogonal component of $p_vxp_v$ inside $H^t_b(v,v)$  is equal to $0$.
The \VNA\ $A$ is isomorphic to $C\ootimes C^\op$.
Let $\xi$ be the orthogonal component of $p_vxp_v$ in $H^t(v)$.
It is a $C$-central vector.
We can prove that the $A$-bimodule $H^t(v)$ is isomorphic to $(L^2(C)\otimes W^t(v)\otimes L^2(C))\otimes L^2(C^\op)$.
In particular, the $C$-bimodule $H^t(v)$ is isomorphic to a direct sum of the coarse bimodule $L^2(C)\otimes L^2(C)$.
Therefore, $\xi=0.$
A similar argument shows that the orthogonal component of $p_vxp_v$ in $H_b(v)$ is equal to $0$.
We obtain that $p_vxp_v$ is in the $A$-bimodule generated by $p_v.$
Therefore, $x$ belongs to $B$.
This proves the claim.

Consider $x\in S\cap T(\J)'$.
Since $A$ is contained in $T(\J)$, we have that $x\in A'\cap S\subset B.$
The element $p_vx$ is in $L^2(Ap_v)$ for any $v\in V^+.$
It commutes with the two elements $$\diagdblecup \text{ and } \diagdblecap.$$
A similar argument to \cite[Corollary 4.11]{JSW} shows that $p_vx\in \C p_v$ for any $v\in V^+.$
Since $x=\sum_{v\in V^+}p_v x$, we obtain that $x\in Z(T).$
In particular, $T(\J)'\cap S=Z(T)$.

Suppose that $x\in S\cap S(\J)'.$
By the argument of above, we have that $x\in Z(T).$
Hence, there exists a bounded function $f:V^+\loriar\C$ such that $x=\sum_{v\in V^+}f(v)p_v.$
Consider the element $\Vert\in D(1,1)\subset S(\J)$ which commutes with $x$.
Observe that $p_v\Vert p_w\neq 0$ if $v=w$ or if $d(v,w)=2$, where $d$ is the length metric of the graph $\Ga.$
Therefore, $p_vx\Vert p_w=f(v) p_v\Vert p_w=p_v\Vert x p_w= f(w) p_v\Vert p_w$ for any $v,w\in V^+$ such that $d(v,w)=2.$
This implies that $f(v)=f(w)$ if $d(v,w)=2.$
Since $\Ga$ is connected, we obtain that $f$ is constant.
Therefore, $x\in \C 1.$
\end{proof}

Part of the following proposition can be deduced from \cite{CJS}.
We provide a full proof for the convenience of the reader.

\begin{prop}\label{prop:SQ-factor}
Consider a \SPA\ $\Ql$ which is contained in $\Pl.$
Then $S_k(\Ql)$ is a II$_1$ factor.
If $v\in\Ve$, then the map $tr:S_k(\Ql)\loriar\C, x\longmapsto \mu_V(v)^{-2} \langle \pi_k(x) p_v,p_v\rangle$ is the unique \NFTS\ of $S_k(\Ql).$
\end{prop}

\begin{proof}
Let $\J$ be the \TLJ\ \PA\ with modulus $\delta$.
We have a chain of inclusions $\J\subset \Ql\subset \Pl$.
This implies that we have the chain of inclusions $S_k(\J)\subset S_k(\Ql) \subset S_k(\Pl)=S_k.$
We obtain that $S_k(\Ql)$ is a factor since $S_k(\J)\subset S_k(\Pl)$ is an irreducible subfactor by Proposition \ref{prop:irreducible_subfactor}.

For any vertex $v\in\Ve$ we consider the linear functional $$tr_v:S_k(\Ql)\loriar\C, x\longmapsto \mu_V(v)^{-2} \langle \pi_k(x) p_v,p_v\rangle.$$
The linear function $tr_v$ is a normal state since it is a vector state.
Denote by $B$ the $*$-algebra $Gr_k\Ql\boxtimes Gr_k\Ql.$

Since $\Ql$ is a \SPA, we have that $\Ql^\ep$ is one dimensional and can be identified to the space of constant functions of $\Pl^\ep\simeq \ell^\infty(\Ve).$
Consider $b\in B$ and the element $E(b)\in\Ql^\ep,$ where $E:Gr_k\Ql\boxtimes Gr_k\Ql\loriar\Pl_0^\ep$ is the map defined in Section \ref{sec:construction}.
Observe that $tr_v(b)$ is the value of $E(b)$ at the vertex $v$.
Since $E(b)$ is constant, we have that $tr_v(b)=tr_w(b)$ for any $v,w\in \Ve.$
By density, we obtain that $tr_v=tr_w$ for any $v,w\in\Ve.$

We fix $v\in \Ve.$ Let us show that $tr:=tr_v$ is tracial.
By density, it is sufficient to show that $tr(ab)=tr(ba)$ for any $a,b\in B.$
Consider $a,b\in B$ and denote by $a_{n,m}$ and $b_{n,m}$ their $(n,m)$-component in $D_k(n,m), n,m\geqslant 0.$
Recall that $Tr$ is the weight of $S_k$ defined in Proposition \ref{prop:vna}.4.
Observe that 
\begin{align*}
tr(ab) & = \sum_{n,m,i,j\geqslant 0} tr(a_{n,m} b_{i,j}) = \sum_{n,m,i,j\geqslant 0} \mu_V(v)^{-2} Tr(a_{n,m}b_{i,j}p_v)\\
 & = \sum_{n,m\geqslant 0} \mu_V(v)^{-2} Tr(a_{n,m}b_{n,m}p_v) \text{ since } a_{n,m}\perp b_{i,j}p_v \text{ in $H_k$ if } (n,m)\neq (i,j)\\
& = \sum_{n,m\geqslant 0} tr(a_{n,m}b_{n,m}).
\end{align*}
Hence, it is sufficient to show that $tr(ab)=tr(ba)$ for $a,b\in D_k(n,m), n,m\geqslant 0$.
We fix $n,m\geqslant 0, a,b\in D_k(n,m), $ and $x,y \in \Ql^+_{2k+n+m}$ such that $j_k(n,m)(x)=a, j_k(n,m)(y)=b.$
Since $\Ql$ is spherical, we have that $tr(ab)=\tau_l(xy)$, where $xy$ is the product of $x$ and $y$ in the planar algebra $\Ql$ and where $\tau_l$ is the left trace of $\Ql.$
By traciality of $\tau_l$, we obtain that $\tau_l(xy)=\tau_l(yx)=tr(ba).$
Therefore, $tr$ is a a normal tracial state of the factor $S_k(\Ql).$
Since $S_k(\Ql)$ is infinite dimensional, we obtain that $S_k(\Ql)$ is a II$_1$ factor and its unique \NFTS\ is $tr.$
\end{proof}

\begin{prop}\label{prop:subalgebras}
Consider a \SPA\ $\Ql$ which is contained in $\Pl.$
The tower of \VNA s $M_0(\Ql)\subset M_1(\Ql)\subset M_2(\Ql)\subset\cdots$ is isomorphic to the tower constructed in \cite{GJS_1}.
The inclusion $T_k(\Ql)\subset S_k(\Ql)$ is isomorphic to the k-th \SEI\ constructed in \cite{CJS}.
In particular, the \SPA\ of $M_0(\Ql)\subset M_1(\Ql)$ is isomorphic to $\Ql$ and the \SEI\ associated to $M_k(\Ql)\subset M_{k+1}(\Ql)$ is isomorphic to $T_{k+1}(\Ql)\subset S_{k+1}(\Ql).$
\end{prop}

\begin{proof}
Let $\Ql$ be a \SPA\ contained in $\Pl.$
We put $B=Gr_k\Ql\boxtimes Gr_k\Ql$.
In \cite{CJS}, to $\Ql$ is associated a tracial $*$-algebra $(V,\wedge,\dag,Tr)$ where $\wedge$ is the multiplication, $\dag$ the anti-linear involution, and $Tr$ a tracial linear functional.
They also consider another tracial $*$-algebra $(W,\star,\dag,Tr')$ where $W$ is a copy of $V$ as a vector space.
They provide a map $X:V\loriar W$ which is an isomorphism of tracial $*$-algebras and consider a projection $p_{k,+}\in V$ that satisfies that $X(p_{k,+})=p_{k,+}$.
Observe that the vector spaces $p_{k,+}Wp_{k,+}$ and $B$ are equal.
Moreover, the $*$-structures of $(p_{k,+}Wp_{k,+},\star,\dag)$ and $(B,\cdot,\dag)$ are defined by the same planar tangles.
If we identify $\Ql_0^\ep$ with the complex numbers inside $\Pl^\ep_0$, then we have that our trace $tr:B\loriar\C$ considered in Proposition \ref{prop:SQ-factor} is equal to the restriction to $B$ of the map $E:Gr_k\Pl\boxtimes Gr_k\Pl\loriar\Pl_0^\ep$ defined in Section \ref{sec:construction}.
Observe that by definition and the fact that $\Ql$ is spherical, we have that $\delta^{-2k} Tr' \vert_{ p_{k,+} W p_{k,+} }$ and $E\vert_B$ are the exact same maps.
Therefore, the tracial $*$-algebras $(p_{k,+}Wp_{k,+},\star,\dag,Tr')$ and $(B,\cdot,\dag,tr)$ are isomorphic.
Thus $(B,\cdot,\dag,tr)$ is isomorphic to $(p_{k,+}Vp_{k,+},\wedge,\dag,Tr)$.
It is proved in \cite{CJS} that $V$ acts by bounded operators on the GNS \HS\ $L^2(V,Tr).$
This implies that the restriction of $\delta^{-2k}Tr$ to $p_{k,+}Vp_{k,+}$ that they denote by $\tau_k\boxtimes\tau_k$ is a faithful tracial state.
Let $L_k$ be the GNS completion of $p_{k,+}Vp_{k,+}$ \WRT\ $\tau_k\boxtimes\tau_k$ (which is denoted by $M_k\boxtimes M_k$ in \cite{CJS}.).
It is proved that $L_k$ is a II$_1$ factor with unique \NFTS\ $\tau_k\boxtimes\tau_k$.

Consider the \VNA\ $S_k(\Ql)$ which is the bicommutant of $B$ inside $S_k$.
The \VNA\ $S_k(\Ql)$ is a II$_1$ factor and its unique \NFTS\ is $tr$.
Therefore, $L_k$ and $S_k(\Ql)$ are both II$_1$ factors which contain a weakly dense $*$-subalgebra $X^{-1}(B)$ and $B$ respectively.
Moreover, $\tau_k\boxtimes\tau_k\circ X^{-1}(b)=tr(b)$ for any $b\in B$.
This implies that there exists an isomorphism $\Phi:L_k\loriar S_k(\Ql)$ such that $\Phi(X^{-1}(b))=b$ for any $b\in B.$
Let $F_k$ be the \VNA\ generated by $Gr_k\Ql$ and $Gr_k\Ql^\op$ inside $L_k.$
It is easy to see that $X^{-1}$ sends $Gr_k\Ql$ and $Gr_k\Ql^\op$ to itself.
This implies that $\Phi(F_k)=T_k(\Ql)$. 
Therefore, $T_k(\Ql)\subset S_k(\Ql)$ is isomorphic to the k-th \SEI\ constructed in \cite{CJS}.

A similar proof shows that the tower of \VNA s $M_0(\Ql)\subset M_1(\Ql)\subset M_2(\Ql)\subset\cdots$ is isomorphic to the tower  constructed in \cite{GJS_1}.
The rest of the proposition follows from \cite[Theorem 8]{GJS_1} and \cite[Theorem 3.3]{CJS}.
\end{proof}

\begin{defi}
If $\Pl$ is a \BGPA\ or a \SPA\ we say that $T_0\subset S_0$ is the \SEI\ associated to $\Pl$ and denote it by $\TS.$
\end{defi}

\section{Fixed point planar algebras}\label{sec:fixedpoint}

\subsection{Actions of the automorphism group of the \BGPA}

Let $(\Ga,\mu)$ be a weighted graph with modulus $\delta>0$.
Consider the \BGPA\ $\Pl$.
An automorphism of $\Pl$ is a sequence of maps $a=(a_n^\pm:n\gO)$ such that $a_n^\pm$ is an automorphism of the \VNA\ $\Pl_n^\pm$ and such that $a$ commutes with the action of the planar tangles.
We denote by $\Aut(\Pl)$ the automorphism group of $\Pl.$

Let $\text{Aut}(\Ga)$ be the group of permutations of the vertices that conserve the number of edges connecting pairs of vertices and send even vertices to even vertices.
We label each n-fold multiple edges by $\{1,\cdots,n\}$ and extend each element of $\text{Aut}(\Ga)$ by a permutation of the edges that preserves the labeling.
Let $\AutG$ be the subgroup of $g\in \Aut(\Ga)$ such that $\mu(ga)=\mu(a)$ for any edge $a$.
We extend the action of $\AutG$ to an action on all paths $C_*$ of $\Ga$ that we denote as follows: $\AutG\times C_*\loriar C_*, (g,a)\mapsto ga.$
Burstein proved that the following map $\gamma:\AutG\times \Pl\loriar\Pl, (g,e_{a,b})\mapsto e_{ga,gb}$ defines an embedding of $\AutG$ into $\AutP$ \cite{Burstein_BGPA}.

Consider the \VNA\ $\Pl_1^+$ and its unitary group $L=U(\Pl_1^+).$
We recall an argument due to Burstein that explains how the group $L$ embeds in $\AutP.$
Consider the map $a\in C_*\mapsto \bar a$ that reverses the orientation of a path.
It induces an injective $*$-morphisms $Rev:\Pl_1^\pm\loriar \Pl_{1}^\mp,e_{a,b}\mapsto e_{\bar a, \bar b}$.
We identify $\Pl_{n}^\pm$ with its image in $\Pl_{n+1}^+$ and $\Pl_n^-$ with its image in $\Pl_{n+1}^+.$
Consider the collection of shift operators $sh:\Pl_n^\pm\loriar \Pl^\pm_{n+2}, e_{a,b}\mapsto \sum_{c\in C_2^\pm} e_{ca,cb}, n\gO.$
If $x\in \Pl^+_1$, we consider the element $x^+_1=x\in \Pl_1^+, x^+_2=x Rev(x)\in \Pl_2^+, x^+_{2n+1}=x^+_{2n} sh^n(x)\in\Pl_{2n+1}^+, x^+_{2n+2}=x^+_{2n+1} sh^n\circ Rev(x)\in \Pl_{2n+2}^+$ and $x_1^-=Rev(x)\in\Pl_1^-, x^-_2=x_1^- sh(x) \in\Pl_2^-, x^-_{2n+1} = x^-_{2n}  sh^n(x_1^-)\in\Pl_{2n+1}^-, x^-_{2n+2}=x^-_{2n+1} sh^{n+1}(x)\in\Pl_{2n+2}^-$ for any $n\geqslant 1$.
Consider $u\in L$, we have that $u^\pm_n$ is a unitary of $\Pl_n^\pm$ for any $n\geqslant 1.$ 
We put $u^\pm_0=1$ and consider the collection of automorphisms $(Ad(u_n^\pm):n\gO)$, where $Ad(u_n^\pm)(a)=u_n^\pm a(u_n^\pm)^*$ for any $n\geqslant 0, a\in \Pl_n^\pm.$
By \cite[Section 3]{Burstein_BGPA}, the map $u\in L\mapsto (Ad(u_n^\pm):n\gO)$ is an embedding of the group $L$ into $\AutP.$

Consider the group $\AutG$ and its action $\gamma$ on $\Pl.$
This action $\gamma$ defines an action of $\AutG$ on $L$ that we continue to denote by $\gamma$.
Consider the semi-direct product $L\rtimes\AutG$ \WRT\ this action.
Burstein proved that those two subgroups generates $\AutP$ and that $\AutP$ is isomorphic to $L\rtimes \AutG$ \cite{Burstein_BGPA}.
We identify $\AutP$ and $L\rtimes\AutG.$
If $g\in\AutP, x\in \Pl_n^\pm$, we denote by $g(x)$ the image of $x$ under the automorphism $g$.

We consider the action of $\AutP$ on $V$ given by 
$$ug \cdot v = gv \text{ for any } u\in\UP,g\in\AutG, \text{ and } v\in V.$$

Here are some remarks regarding \FPS s contained in $\Pl$.
\begin{remark}
Consider a subgroup $G<\AutP$ and the \FPS\ $\Ql=\Pl^G$ under the action of $G.$
\begin{itemize}
\item The collection of \FPS s $(\Ql_n^\pm=(\Pl^\pm_n)^G:n\geqslant 0)$ is a planar algebra.
\item If $G$ is contained in $\AutG$ and acts transitively on $V^+$ and $V^-$, then $\Ql_0^+$ and $\Ql_0^-$ are one dimensional and $\Ql$ satisfies all the axioms of a \SPA\ except the sphericality.
In fact, $\Ql$ is a \SPA\ \IFF\ for any $a\in C_1^\pm$ we have that $\mu(a)\vert\{\alpha\in G\cdot a:s(\alpha)=v^\pm\}\vert=\mu(\bar a)\vert\{\alpha\in G\cdot a:t(\alpha)=v^\mp\}\vert$ where $v^\pm\in V^\pm$ is a fixed pair of vertices.
Indeed, the planar algebra $\Ql$ is spherical \IFF\ $\tau_r$ and $\tau_l$ coincide on $\Ql_1^\pm$.
Recall that $\{e_{a,b}:(a,b)\in ST_1^\pm\}$ is a \SMU\ of $\Pl_1^\pm,$ where $ST^\pm_1=\{(a,b)\in C_1^\pm\times C_1^\pm : s(a)=s(b) \text{ and } t(a)=t(b) \}.$
Therefore, $\Ql_1^\pm$ is equal to the weak closure of $\text{Span} \{  f_{a,b} = \sum_{(\alpha,\beta)\in G\cdot (a,b)} e_{\alpha,\beta}: (a,b)\in ST_1^\pm\}.$
Consider $(a,b)\in ST_1^\pm.$
We have that 
\begin{align*}
\tau_l(f_{a,b}) & = \sum_{ (\alpha,\beta) \in G\cdot (a,b) } \tau_l( e_{ \alpha,\beta } ) = \delta_{a,b} \sum_{ \alpha\in G\cdot a } \mu(\bar a) e_{ t(\alpha) } = \delta_{ a,b } \sum_{ w\in V^\mp } \sum_{ \alpha \in G\cdot a : t( \alpha ) =w } \mu(\bar a) e_w\\
& = \delta_{a,b} \mu(\bar a) \vert\{\alpha\in G\cdot a:t(\alpha)=v^\mp\}\vert.
\end{align*}
A similar computation shows that $\tau_r(f_{a,b})=\delta_{a,b}\mu(a)\vert\{\alpha\in G\cdot a:s(\alpha)=v^\pm\}\vert.$
\item If $\Pl^G$ is a \SPA, then $G$ acts transitively on $V^+$ and $V^-$.
\item In general $\Ql=\Pl^G$ is reducible.
We can deduce the following.
Suppose there exists a subgroup $G<\AutP$ such that the \FPS\ $\Pl^G$ is an irreducible \SPA.
Then, the group $\AutG$ acts transitively on the set of positives edges $C_1^+.$
This implies that the weight $\mu$ is constant on $C_1^+.$
\item If we start with a countable locally finite undirected connected bipartite graph $\Ga$ that can have multiple edges between two vertices and a group $G<\Aut(\Ga)$ that acts transitively on $V^+$ and on $V^-$, then it was observed in \cite[Proposition 2.5]{Arano_Vaes_SF} that there exists a unique weight $\mu:C^1\loriar\R_+^*$ such that $\Pl^G$ is a \SPA, where $\Pl$ is the \BGPA\ associated to $(\Ga,\mu).$
\end{itemize}
\end{remark}

We fix a natural number $k\gO.$

\begin{prop}\label{prop:minimal}
Let $j_k(n,m):\Pl^+_{n+m+2k}\loriar S_k$ be the inclusion of the vector space $\Pl^+_{n+m+2k}$ in $S_k$ for $n,m\geqslant 0.$
There exists a group morphism $ \sigma :\Aut(\Pl) \loriar \Aut(S_k)$ that satisfies that $\sigma_g\circ j_k(n,m)(x)=j_k(n,m)(g(x))$ for any $g\in \Aut(\Pl),n,m\gO,$  and $x\in \Pl^+_{n+m+2k}.$
In particular, $\Aut(\Pl)$ acts on $T_k$, $M_k$ and $M_k^\op.$
Moreover, the action of $\sigma$ is minimal, i.e. the relative commutant $(S_k^{\Aut(\Pl)})'\cap S_k$ is trivial.
\end{prop}

\begin{proof}
For any $g\in \Aut(\Pl),n,m\gO,x\in \Pl^+_{n+m+2k}$, we put $\sigma_g( j_k(n,m)(x))=j_k(n,m)(g(x)).$
Since any element $g\in\AutP$ commutes with the action of the planar operad, we obtain that $\sigma_g$ is a $*$-algebra morphism of $\GrkPb.$
Observe that $p_v H_k$ is orthogonal to $p_wH_k$ for two different vertices $v,w\in V^+.$
Since $\sum_{v\in V^+}p_v=1$, we obtain that
\begin{equation}\label{equa:decompoH} H_k=\bigoplus_{v\in V^+} p_vH_k.\end{equation}
Consider $g\in\AutG$ and $x\in \GrkPb$ such that $\{v\in V : p_vx\neq 0\}$ is finite.
Note that $Tr(xx^*)<\infty.$
Observe that $Tr\circ\sigma_g= c_g \cdot Tr$ for any $g\in G$, where $c_g$ is the ratio $\frac{Tr(p_{gw})}{Tr(p_w)}=\frac{\mu_V(gw)^2}{\mu_V(w)^2}$ which does not depend on $w\in V^+.$
We put $U_g(x) = \sigma_g(x)/\sqrt{c_g}$.
Observe that 
$$c_g \Vert U_g(x)\Vert_2^2 = \Vert \sigma_g(x)\Vert_2^2 = \tau_V\circ E\circ \sigma_g(xx^*) = \tau_V \circ \sigma_g \circ E(xx^*),$$
since $g$ commutes with the action of the planar operad.
Since $\tau_V\circ \sigma_g = c_g\cdot \tau_V,$ we obtain that $\Vert U_g(x)\Vert_2^2=\tau_V\circ E(xx^*)=\Vert x\Vert_2^2.$
This implies that $U_g$ extends as an isometry of $H_k$.
By definition, $U_{ g^{-1} }\circ U_g(x)=x$ for any $n,m\gO$ and $x$ in the range of $j_k(n,m).$
This implies that $U_{ g^{-1} }\circ U_g$ is the identity operator.
By symmetry, $U_g$ is a unitary of $H_k$ and $U_g^*=U_{g^{-1}}.$
By definition and by a density argument we obtain that $U_g U_h=U_{gh}$ for any $g,h\in\AutG.$
Consider the map $Ad(U_g)(x)=U_gxU_g^*$ for any $g\in \AutG, x\in B(H_k).$
Consider $x\in\GrkPb,g\in\AutG,$ and $\xi\in H_k\cap \GrkPb$.
We have that $$Ad(U_g)(x)\xi  = U_gxU_g^*\xi = U_g x \sqrt{c_g}\sigma_{g^{-1}}(\xi) =  \sigma_g( x \sigma_{ g^{-1} } ( \xi ) )=\sigma_g(x)\xi.$$
If $g\in L$, we define $U_g(x)=\sigma_g(x)$ for any $x\in H_k\cap \GrkPb.$
A similar proof shows that $Ad(U_g)(x)=\sigma_g(x)$ for any $x\in \GrkPb,g\in L.$
One can check that $U_g U_l U_g^*=U_{\gamma_g(l)}$ for any $g\in\AutG,l\in L.$
This implies that $U:\AutP\loriar U(H_k)$ is a unitary representation.
By density, we obtain that $Ad(U_g)(S_k)=S_k$ for any $g\in\AutP.$
We denote by $\sigma_g$ the restriction to $S_k$ of the automorphism $Ad(U_g), g\in\AutP.$
We have that $\sigma_g\circ j_k(n,m)(x)=j_k(n,m)(g(x))$ for any $g\in \Aut(\Pl),n,m\gO,$  and $x\in \Pl^+_{n+m+2k}.$
Hence, $\Aut(\Pl)$ acts on $T_k$, $M_k$ and $M_k^\op$ by restricting $\sigma_g,g\in\AutP.$

If $\mathcal J$ is the \TLJ\ \PA\ with modulus $\delta$, then $S_k(\mathcal J)$ is contained in the \FPS\ $S_k^{\Aut(\Pl)}.$
By Proposition \ref{prop:irreducible_subfactor}, $S_k(\mathcal J)$ is an irreducible subfactor of $S_k$.
Therefore, $S_k^{\Aut(\Pl)}\subset S_k$ is an irreducible subfactor.
\end{proof}

\begin{prop}\label{prop:fixedpoint}
Consider a subgroup $G<\AutP$ and the \FPPA\ $\Ql:=\Pl^G.$
Assume that $\Ql$ is a \SPA.
Consider the fixed point \VNA s $S_k^G, M_k^G$, and $(M_k^\op)^G$.
We have the equalities $S_k^G=S_k(\Ql), M_k^G=M_k(\Ql)$, and $(M_k^\op)^G=M_k^\op(\Ql),$
where $S_k(\Ql),M_k(\Ql)$, and $M_k^\op(\Ql)$ are the von Neumann subalgebras of $S_k$ defined in Section \ref{sec:subalgebras}.
In particular, $M_0^G\subset M_1^G$ is a \SF\ with \SPA\ isomorphic to $\Ql$.
Furthermore, the \SEI\ of $\Ql$ is isomorphic to $M_0^G\vee (M_0^\op)^G \subset S_0^G.$
\end{prop}

\begin{proof}
We assume that $k=0$ and drop the subscript $k$.
The general case can easily be deduced.
Let us show that $S^G=S(\Ql).$
By definition, $S(\Ql)$ is the weak closure of $B:=Gr\Ql\boxtimes Gr\Ql$ inside $S$.
Any element of $B$ is $G$-invariant.
Therefore, $B$ is included in $S^G$ and so does its weak closure $S(\Ql).$

Consider a vertex $o\in V^+$ and the normal state $tr:S\loriar\C, x\longmapsto Tr(p_o)^{-1}Tr(xp_o).$
By Proposition \ref{prop:SQ-factor}, $S(\Ql)$ is a II$_1$ factor and $tr$ is its unique \NFTS.
In particular, $tr$ is faithful on $S(\Ql).$
Let us show that $tr$ is faithful on $S^G.$
Consider $x\in S^G$ such that $tr(xx^*)=0.$
The action of $G$ on $V^+$ is transitive since $\Pl^G$ is a \SPA.
Hence, for any $v\in V^+$, there exists $g_v\in G$ such that $g o=v.$
Observe that 
\begin{align*}
Tr(xx^*) & = \sum_{v\in V^+} Tr(xx^* p_v) = \sum_{v\in V^+} Tr(xx^* p_{g_v o})\\
& = \sum_{v\in V^+} Tr(\sigma_{g_v}(xx^* p_o)) \text{ since } xx^* \text{ is $G$-invariant}\\
& = \sum_{v\in V^+} \frac{\mu_V(v)^2}{\mu_V(o)^2} Tr(xx^* p_o) = \sum_{v\in V^+} \mu_V(v)^2 tr(xx^*) = 0.
\end{align*}
Since $Tr$ is faithful, we obtain that $x=0.$
Therefore, $tr$ is a \NFS\ on $S^G.$

Let us show that $L^2(S(\Ql),tr)=L^2(S^G,tr),$ where $L^2(S(\Ql),tr)$ and $L^2(S^G,tr)$ are the GNS \HS\ associated to $(S(\Ql),tr)$ and $(S^G,tr).$
Consider $\xi$ in the orthogonal complement of $L^2(S(\Ql),tr)$ inside $L^2(S^G,tr).$
Observe that $Tr((p_o\xi)(p_o\xi)^*)=Tr(\xi\xi^*p_o)=Tr(p_o) tr(\xi\xi^*)<\infty.$
Therefore, $p_o\xi\in L^2(S,Tr).$
The \HS\ $L^2(S,Tr)$ is equal to the direct sum 
$$\bigoplus_{n,m\geqslant 0} H(n,m).$$
Hence, there exists $x^n_m\in H(n,m),n,m\geqslant 0$ such that $p_0\xi=\sum_{n,m\geqslant 0} x^n_m$ where the sum converges in $L^2(S,Tr)$.
Let us fixed $n,m\geqslant 0.$
We have that $x^n_m\in p_oH(n,m).$
Observe that $$\{j(n,m)(e_{a,b}): (a,b)\in ST_{n+m}^+, s(a)=o\}$$ is an orthogonal basis of $p_oH(n,m).$
Since $\Ga$ is locally finite, we obtain that this basis is finite.
Therefore, $p_oH(n,m)$ is finite dimensional and is equal to $p_oD(n,m)$.
Hence, there exists a unique $d\in e_o\Pl_{n+m}^+$ such that $j(n,m)(d)=x^n_m$.
Recall that 
$$\Pl_{n+m}^+=\bigoplus_{v\in V^+, w\in V} B(\ell^2(C_{n+m}^+(v,w))), \text{ where } C_{n+m}^+(v,w)=\{a\in C_{n+m}^+: s(a)=v,t(a)=w\}.$$
We have that $d\in \bigoplus_{w\in V} B(\ell^2(C_{n+m}^+(o,w))).$

If $g\in \AutP,$ then $g(d)\in \bigoplus_{w\in V} B(\ell^2(C_{n+m}^+(go,w))).$
Consider $G_o=\{g\in G:go=o\}$, and a \SR\ $\lGGor$ of the quotient space $G/G_o$.
The elements $p_o$ and $\xi$ are $G_o$-invariants.
Therefore, $p_o\xi$ is $G_o$-invariant as an element of $L^2(S,Tr)$.
Consider the orthogonal projection $P:L^2(S,Tr)\loriar H(n,m)$ and observe that $P$ is $G$-equivariant.
Consider $g\in G_o$.
Since $g$ fixes $o$, we have that $Tr\circ \sigma_g = Tr$ by uniqueness of the trace.
Hence, $\sigma_g$ and $U_g$ coincide on $L^2(S,Tr)\cap S$, where $U_g$ is the unitary defined in the proof of the previous proposition.
Observe that $j(n,m)(g(d)) = \sigma_g \circ j(n,m)(d) = U_g(P(p_o\xi)) = P\circ U_g(p_o \xi) = P(p_o \xi ) = j(n,m)(d).$
Therefore, $d$ is $G_o$-invariant as an element of $\Pl_{n+m}^+$.

We have that $go\neq ho$ for any $g\neq h$ in $\lGGor.$
This implies that the elements of $\{g(d):g\in\lGGor\}$ belongs to distinct summands of the \VNA\ $$\Pl_{n+m}^+=\bigoplus_{v\in V^+, w\in V} B(\ell^2(C_{n+m}^+(v,w))).$$
Moreover, $\Vert g(d)\Vert=\Vert d\Vert$ for any $g\in\AutP$ since $g$ is an automorphism of the \VNA\ $\Pl_{n+m}^+.$
This implies that the sum $\sum_{g\in\lGGor} g(d)$ converges for the \SOT\ to an element $\tilde d\in\Pl_{n+m}^+.$
Let us show that $\tilde d$ is $G$-invariant.
Consider $f=\sum_{r\in R} r(d)$, where $R$ is a \SR\ of $G/G_o$.
For any $r\in R$ there exists $g_r\in \lGGor, h_r\in G_o$ such that $r=g_r h_r.$
Moreover, $\{g_r:r\in R\}=\lGGor.$
Observe that 
$$f=\sum_{r\in R} r(d) = \sum_{r\in R} g_r h_r(d) = \sum_{r\in R} g_r(d) = \sum_{g\in \GGo} g(d) = \tilde d.$$
Therefore, $\tilde d$ does not depend on the choice of the \SR\ $\lGGor.$
Consider $g\in G.$
We have that $g(\tilde d) = \sum_{h\in\lGGor} gh(d) = \sum_{r\in R} r(d),$
where $R=\{gh:h\in\lGGor\}$.
The set $R$ is a \SR\ of $G/G_o.$
Therefore, $g(\tilde d)=\tilde d.$
Hence, $\tilde d$ is a $G$-invariant element of $\Pl_{n+m}^+.$
This means that $\tilde d\in\Ql_{n+m}^+.$

We put $y^n_m= j(n,m)(\tilde d).$
By definition, $y^n_m\in B.$
By assumption, $tr(\xi (y^n_m)^*)=0.$
Observe that 
\begin{align*}
Tr(p_o) tr(\xi\xi^*) & = Tr(\xi\xi^*p_o)=Tr(\xi (p_o\xi)^*)=\sum_{k,l\geqslant 0} Tr(\xi(x^k_l)^*)\\
& = \sum_{k,l\geqslant 0} Tr(\xi (y^k_l)^*p_o) \text{ since } p_oy^k_l=x^k_l\\
& = Tr(p_o) \sum_{k,l\geqslant 0} tr(\xi (y^k_l)^*)=0.
\end{align*}
We obtain that $L^2(S(\Ql),tr)=L^2(S^G,tr)$ since the orthogonal complement of $L^2(S(\Ql),tr)$ inside $L^2(S^G,tr)$ is trivial.
Consider $\xi\in S^G$.
Since $tr$ is a faithful state, we have that $\xi\in L^2(S^G,tr)$.
Hence, $\xi\in L^2(S(\Ql),tr)$.
Let us show that $\xi$ is a bounded vector of the II$_1$ factor $S(\Ql).$
Consider the map $\theta:S(\Ql)\loriar L^2(S(\Ql),tr),y\longmapsto y\xi.$
Observe that 
\begin{align*}
tr(\theta(y)\theta(y)^*) & = Tr(p_o)^{-1} Tr(y\xi\xi^* y^* p_o)\leqslant Tr(p_o)^{-1} \Vert \xi\xi^*\Vert_S Tr(yy^*p_o)\\
& \leqslant \Vert \xi\Vert_S^2 tr(yy^*) \text{ for any } y\in S(\Ql).
\end{align*}
Therefore, $\theta$ extends as a bounded linear operator from $L^2(S(\Ql),tr)$ to $L^2(S(\Ql),tr)$.
Since $S(\Ql)$ is a II$_1$ factor and $tr$ is its unique \NFTS, we obtain that $\xi\in S(\Ql)$ by \cite[Theorem 1.2.4.2]{Jones_Sunder_subfactors}.
Therefore, $S(\Ql)=S^G.$

A similar proof show that $M(\Ql)=M^G$ and $M(\Ql)^\op=(M^\op)^G.$
By \cite{GJS_1} and Proposition \ref{prop:subalgebras}, $\Ql$ is isomorphic to the \SPA\ of $M_0(\Ql)\subset M_1(\Ql).$
Therefore, it is isomorphic to the \SPA\ of $M_0^G\subset M_1^G.$
The \SEI\ of $\Ql$ is isomorphic to $M_0(\Ql)\vee M_0(\Ql)^\op \subset S_0(\Ql)$ by Proposition \ref{prop:subalgebras}.
Therefore, it is isomorphic to $M_0^G\vee (M_0^\op)^G\subset S_0^G.$
\end{proof}

\subsection{Proof of Theorem \ref{theo:introone}}
Let $(\Ga,\mu)$ be a weighted graph with modulus $\delta$ and let $G<\AutP$ be a countable or closed subgroup.
Let $\Pl$ be the \BGPA\ associated to $(\Ga,\mu)$ and let $\TS$ be its \SEI.
Consider the \FPS\ $\Ql=\Pl^G$ that we assume to be a \SPA.
This implies that the action of $G$ on $V^+$ is necessarily transitive and that $\TS$ is an irreducible \SF.
If $K<L$ is an inclusion of groups, we denote by $\langle L/K\rangle$ (resp. $\langle K\backslash L/K\rangle$) a \SR\ of the coset space $L/K$ (resp. $K\backslash L/K$).
If $v,w\in V^+$, then $G_v$ (resp. $G_{v,w}$) denotes the subgroup of $G$ that fixes the vertex $v$ (resp. fixes the vertices $v$ and $w$).
Let $o\in V^+$ be a fixed vertex.
Note that $G/G_o$ is in bijection with $V^+$ since $G$ acts transitively on $V^+.$
Hence, $\{p_{go}:g\in\lGGor\}$ is the set of minimal central projections of $T$.

\begin{remark}\label{rem:almost-normal}
\begin{itemize}
\item The subgroup $G_o<G$ is \AN.
Indeed, consider $g\in G.$
We need to show that $G_o\cap gG_og^{-1}$ is a finite index subgroup of $G_o$ and $gG_og^{-1}.$
Observe that $G_o\cap gG_og^{-1}=\Gogo=\{h\in G : hgo=go \text{ and } ho=o\}.$
Consider the map $h\in G_o\longmapsto hgo\in V.$
We have that $G_o /\Gogo$ is in bijection with the orbit $G_o\cdot go.$
But $G_o\cdot go$ is contained in the sphere centered in $o$ with radius equal to the distance between $o$ and $go.$
This sphere is a finite set since $\Ga$ is locally finite.
Therefore, $\Gogo$ is a finite index subgroup of $G_o$.
A similar argument shows that $\Gogo$ is a finite index subgroup of $gG_og^{-1}$.
Therefore, $G_o<G$ is an \AN\ subgroup.
\item By Propositions \ref{prop:SQ-factor} and \ref{prop:fixedpoint}, the unique \NFTS\ of the II$_1$ factor $S^G$ is $tr:S^G\loriar\C, x\mapsto Tr(p_v)^{-1}Tr(p_vx)$ where $v\in V^+.$
\end{itemize}
\end{remark}

The next proposition gives a decomposition of the $T^G$-bimodule $L^2(S^G,tr)$.

\begin{prop}\label{prop:bimodules}
We define the \HS\ $H_g$ which is the closure in $L^2(S,Tr)$ of $p_o S^G p_{g o}$ for any $g\in G.$
\TFAT
\begin{enumerate}
\item The \HS\ $H_g$ is a bifinite $T^G$-bimodule  for any $g\in G.$
\item The $T^G$-bimodules $L^2(S^G,tr)$ and $\bigoplus_{g\in\GGG} H_g$ are isomorphic.
\end{enumerate}
\end{prop}

\begin{proof}
Proof of (1).
Consider $g\in G$.
Since $Tr(p_o)<\infty$, we have that $p_o S$ is contained in $L^2(S,Tr).$
Therefore, $H_g$ is well defined.
Since $p_{g o}$ and $p_o$ commute with $T^G$, we have that $H_g$ is a $T^G$-bimodule.
Suppose that $Tr(p_o)\leqslant Tr(p_{g o}).$
Since $S^\Gogo$ is a factor by Proposition \ref{prop:minimal}, there exists some elements $m_1,\cdots,m_n \in S^\Gogo$ such that $m_1m_1^*=p_o, m_im_i^*\leqslant p_o, 2\leqslant i \leqslant n$ and $\sum_{j=1}^n m_j^*m_j=p_{go}$.
Since $G_o<G$ is an \AN\ subgroup by Remark \ref{rem:almost-normal}, we have that the index $[G_o:\Gogo]$ is finite.
Therefore, $[T(o)^{\Gogo}: T(o)^{G_o}]<\infty.$
Hence, there exists a finite subset $F\subset T(o)^\Gogo$ such that $T(o)^\Gogo=\spann(T(o)^{G_o}\cdot F).$
Similarly, there exists a finite subset $E\subset T(g o) ^{\Gogo}$ such that $T(g o)^{\Gogo} = \spann(E\cdot T(g o)^{G_{ g o } }).$
Observe that 
$$p_{g o} T^G = T^G p_{g o} = T(g o)^{G_{g o } }.$$
Consider $x\in p_o S^G p_{go}.$
We have that $x= m_1(m_1^*x)$ and $m_1^* x\in p_{go} S p_{go}=T(go).$
Moreover, $m_1^*x$ is $\Gogo$-invariant.
Therefore, $m_1^*x$ is in $T(g o)^{\Gogo}= \spann(E \cdot T(g o)^{G_{go}}).$
Hence, $x\in \spann(m_1\cdot E \cdot T(g o)^{G_{go}})=\spann( m_1\cdot E\cdot T^G).$

Consider $1\leqslant i \leqslant n$.
Observe that $x=\sum_{j=1}^n x m_j^* m_j$ and $xm_i^*\in p_o S p_o=T(o).$
Moreover, $xm_i^*$ is $\Gogo$-invariant.
Hence, $xm_i^*$ is in $T(o)^{\Gogo}= \spann(T(o)^{G_o}\cdot F)= \spann(T^G\cdot F).$
Therefore, $x$ is in $\spann(T^G \cdot F \cdot \{m_j : 1\leqslant j\leqslant n\}).$
Therefore, $H_g$ is a bifinite $T^G$-bimodule.

Proof of (2).
Consider the map $\beta: L^2(S^G,tr)\loriar L^2(S,Tr)$ such that 
$$\beta(x)=\sum_{g\in\GGG} p_o x p_{g o} \sqrt{\frac{[G_o : \Gogo]}{Tr(p_o)}}$$ for any $x\in S^G$, where the sum converges for the $L^2$-norm.
Let us show that $\beta$ is an isometry.
Let $\sigma:G\loriar\Aut(S)$ be the action of $G$ on $S$.
Observe that $Tr\circ\sigma_g= c_g \cdot Tr$ for any $g\in G$, where $c_g$ is the ratio $\frac{Tr(p_{gw})}{Tr(p_w)}$ which does not depend on $w\in V^+.$
In particular, 
\begin{equation}\label{equa:Trsigmag} Tr\circ \sigma_g=Tr \text{ if } g\in G \text{ fixes a vertex of } V^+.\end{equation}
Consider $x\in S^G.$
We have that 
\begin{align*}
Tr(\beta(x)\beta(x)^*) & = \sum_{g,h\in\GGG} \frac{\sqrt{ [G_o : \Gogo ] [G_o : G_{o,ho} ]}}{Tr(p_o)} Tr( p_o x p_{ go } p_{ h o } x^* p_o )\\
 & = \sum_{g\in\GGG} \frac{[G_o : \Gogo ]}{Tr(p_o)}  Tr(p_o x p_{go} x^*).
\end{align*}
Consider $g,h\in G$ such that $G_o g G_o=G_o h G_o$.
Then, there exists $k\in G_o$ such that $g G_o=kh G_o.$
Observe that
\begin{align*}
Tr(p_o x p_{go} x^*) & = Tr(p_o x p_{kho} x^*) = Tr\circ \sigma_{k}(p_o x p_{h o} x^*)\\
 & = Tr(p_o xp_{ho}x^*) \text{ by \eqref{equa:Trsigmag}.}
\end{align*}
Therefore, $Tr(p_o x p_{go} x^*) = [G_o : \Gogo]^{-1} \sum_{k\in \langle G_o/\Gogo\rangle} Tr(p_o x p_{kgo} x^*).$
Hence,
\begin{align*}
Tr(\beta(x)\beta(x)^*)  & = \sum_{g\in\GGG}\sum_{k\in\langle G_o/\Gogo\rangle}\frac{Tr(p_oxp_{kgo}x^*)}{Tr(p_o)}\\
 & = \sum_{s\in \lGGor} \frac{Tr(p_oxp_{so}x^*)}{Tr(p_o)} \text{ for a \SR\ } \lGGor\\ 
 & = \frac{Tr(p_oxx^*)}{Tr(p_o)} = tr(xx^*).
\end{align*}
Therefore, $\beta$ defines an isometry from $L^2(S^G,tr)$ to $L^2(S,Tr).$

The map $\beta$ is $T^G$-bimodular because $p_o$ and $p_{go}$ commute with $T^G$ for any $g\in G.$
Note that $H_g$ is orthogonal to $H_k$ if $g G_o\neq k G_o$.
Therefore, the $T^G$-bimodules $\{H_l : l\in\GGG\}$ are pairwise orthogonal.
We have that $Im\beta\subset \bigoplus_{l\in\GGG} H_l$ by definition.
Let us show that the range of $\beta$ is equal to $\bigoplus_{l\in\GGG} H_l$.

We fix $g\in \GGG.$
Consider $y\in p_o S^G p_{go}=H_g\cap S.$
Let us show that the sum $\sum_{ s\in\lGGogor} \sigma_s(y)$ converges for the \WOT\ of $S$.
Consider $\varepsilon>0,\xi,\eta\in L^2(S,Tr)$ such that $\Vert \xi\Vert_2=\Vert \eta\Vert_2=1.$
If $A\subset V^+$, we denote by $A^c$ its complement and by $p_A$ the projection $\sum_{v\in A} p_v.$
Since $1=\Vert\xi\Vert_2^2=\sum_{v\in V^+} \Vert p_v \xi\Vert_2^2=\sum_{v\in V^+} \Vert p_v \eta\Vert_2^2$, there exists a finite subset $A\subset V^+$ such that $\Vert p_{A^c} \xi\Vert_2^2, \Vert p_{A^c} \eta\Vert_2^2<\varepsilon.$
Consider the set $E=\{s\in \lGGogor : so\in A \text{ or } sgo\in A\}.$
Note that its cardinal $\vert E \vert$ is smaller than $\vert A\vert ([G_o : \Gogo ] + [G_{go} : \Gogo]) <\infty.$
Denote by $E^c$ the complement of $E$ inside  $\lGGogor.$
Observe that 
\begin{align*}
\vert\langle \sum_{s\in E^c} \sigma_s(y)\xi,\eta\rangle\vert & \leqslant  \sum_{s\in E^c} \vert \langle \sigma_s(y) \xi, \eta \rangle \vert \leqslant \sum_{s\in E^c} \vert \langle \sigma_s(y)p_{sgo}\xi,p_{so}\eta\rangle\vert \text{ since } y\in p_o S p_{go}  \\
 & \leqslant \sum_{s\in E^c} \Vert \sigma_s(y) p_{sgo} \xi \Vert_2 \Vert p_{so}\eta\Vert_2 \text{ by the Cauchy-Schwarz inequality}\\
 & \leqslant \sum_{s\in E^c} \Vert y \Vert \Vert p_{sgo} \xi \Vert_2 \Vert p_{so} \eta \Vert_2\\
 & \leqslant \Vert y \Vert \sqrt{ \sum_{s\in E^c} \Vert p_{sgo} \xi \Vert_2^2 } \sqrt{ \sum_{t\in E^c} \Vert p_{to} \eta \Vert_2^2 } \text{ by the Cauchy-Schwarz inequality}\\
 & \leqslant \Vert y \Vert \sqrt{\sum_{w\in A^c} \vert\{ s\in E^c : sgo=w \}\vert \Vert p_w\xi\Vert_2^2} \sqrt{\sum_{w\in A^c} \vert\{ t\in E^c : to=w \}\vert \Vert p_w\eta\Vert_2^2} 
\end{align*}
Consider $w\in V^+, r\in G$ such that $rw=o$ and $F=\{t\in E^c : to=w\}.$
Observe that $r\cdot F\subset G_o\cap r\lGGogor.$
Therefore, $\vert F\vert=\vert r\cdot F\vert\leqslant \vert G_o\cap r\lGGogor\vert=[G_o : \Gogo],$ since $\vert G_o\cap R\vert =[G_o : \Gogo]$ for any \SR\ $R$ of $G/\Gogo.$
Similarly, $\vert\{ s\in E^c : sgo=w \}\vert\leqslant [G_{go} : \Gogo].$
Therefore, 
\begin{align*}
\vert\langle \sum_{s\in E^c} \sigma_s(y)\xi,\eta\rangle\vert & \leqslant \sqrt{ [G_o : \Gogo] [G_{go} : \Gogo]} \Vert p_{A^c}\xi\Vert_2 \Vert p_{A^c}\eta\Vert_2\\
 & \leqslant \varepsilon \sqrt{ [G_o : \Gogo] [G_{go} : \Gogo]}.
\end{align*}
This implies that for any $\delta>0$ there exists a finite set $E\subset\lGGogor$ such that for any sets $I,J$ that contains $E$ we have that 
$\vert \langle \sum_{s\in I} \sigma_s(y)\xi,\eta\rangle - \langle \sum_{t\in J} \sigma_t(y)\xi,\eta\rangle \vert <\delta.$
Hence, $( (\sum_{s\in I} \sigma_s(y) : I\subset \lGGogor)$ is a Cauchy filter in $S$ \WRT\ the weak uniform structure.
This implies that the sum $\sum_{s\in \lGGogor} \sigma_s(y)$ converges weakly in $S$, since $S$ is complete for the weak uniform structure.
We denote by $\Theta_g(y)$ this limit.
Let us show that $\Theta_g(y)$ is $G$-invariant.
Suppose that $R$ is a \SR\ of $G/\Gogo.$
For any $r\in R$ there exists a unique $a_r\in\lGGogor$ and $b_r\in\Gogo$ such that $r=a_rb_r.$
we have that $\sigma_r(y)=\sigma_{a_rb_r}(y)=\sigma_{a_r}\circ\sigma_{b_r}(y)=\sigma_{a_r}(y)$ since $y$ is $\Gogo$-invariant.
Moreover, $\{a_r: r\in R\}=\lGGogor.$
Therefore, $ \Theta_g(y)=\sum_{r\in R} \sigma_r(y).$
Hence, $\Theta_g$ does not depend on the choice of the \SR\ of $G/\Gogo.$
Consider $a\in G.$
We have that $\sigma_a\circ\Theta_g(y)=\sum_{r\in R}\sigma_r(y)$, where $R=a\cdot\lGGogor.$
Since $R$ is a \SR, we obtain that $\sigma_a\circ\Theta_g(y)=\Theta_g(y).$
Hence, $\Theta_g(y)$ is $G$-invariant.

Let $\langle G_o/\Gogo\rangle$ be a \SR\ of $G_o/\Gogo$.
Observe that 
\begin{align*}
tr(\Theta_g(y)\Theta_g(y)^*) & = Tr(p_o)^{-1} Tr(\Theta_g(y)\Theta_g(y)^*p_o) \\
 & = Tr(p_o)^{-1} \sum_{s,t\in \lGGogor} Tr(\sigma_s(y)\sigma_t(y)^* p_o) \\
& = Tr(p_o)^{-1} \sum_{s\in G_o\cap\langle G/\Gogo\rangle} Tr(\sigma_s(y)\sigma_s(y)^*) \text{ since } y\in p_oSp_{go}\\
 & = Tr(p_o)^{-1}  \sum_{s\in G_o\cap\langle G/\Gogo\rangle} Tr(yy^*) \text{ by \eqref{equa:Trsigmag} }\\
 & = \frac{[G_o : \Gogo ] }{ Tr(p_o) } \Vert y \Vert_2^2.
\end{align*}
Hence, $\Theta_g$ extends to an injective bounded linear operator from $H_g$ to $L^2(S^G,tr).$
Observe that $\Theta_g$ is a $T^G$-bimodular map.
Consider $z\in S^G$ and $y=p_o z p_{go}.$
We have that
\begin{align*}
\beta\circ\Theta_g(y) & = \sum_{t\in\GGG} \sum_{s\in\langle G/\Gogo\rangle} p_o \sigma_s(y) p_{to} \sqrt{\frac{[G_o : G_{o,to} ] }{Tr(p_o)}}\\
 & = \sum_{t\in\GGG} \sum_{s\in\langle G/\Gogo\rangle} \delta_{o,so} \delta_{sgo,to} p_o z p_{to} \sqrt{\frac{[G_o : G_{o,to} ] }{Tr(p_o)}},
\end{align*}
where $\delta_{a,b}$ is the Kronecker symbol.
Observe that if $o=so$ and $sgo=to$, then $s\in G_o$ and $G_o g G_o =G_o t G_o.$
Therefore,
\begin{align*}
\beta\circ\Theta_g(y) & = \sum_{s\in G_o\cap \langle G/\Gogo\rangle} \delta_{sgo, go} p_o z p_{go} \sqrt{\frac{[G_o : G_{o,go} ] }{Tr(p_o)}}\\
 & = y \sqrt{\frac{[G_o : G_{o,go} ] }{Tr(p_o)}}
\end{align*}
We obtain that $\beta\circ\Theta_g=\sqrt{\frac{[G_o : G_{o,go} ] }{Tr(p_o)}} Id_{H_g}$ by a density argument.
Therefore, $H_g$ is contained in $Im\beta$ for any $g\in\GGG.$
Since $Im\beta\subset \bigoplus_{l\in\GGG} H_l$ and $\beta$ is an isometry, we obtain that $Im\beta=\bigoplus_{l\in\GGG} H_l$.
Hence, $\beta$ realizes an isomorphism of $T^G$-bimodules from $L^2(S^G,tr)$ onto $\bigoplus_{l\in\GGG} H_l$.
\end{proof}

\begin{prop}\label{prop:maps}
Let $f:G\loriar\C$ be a unital bounded $G_o$-bi-invariant map.
Denote by $\bar f:G/G_o\loriar\C$ the induced map on the coset space.
\TFAT
\begin{enumerate}
\item There exists a unique normal trace-preserving unital $T$-bimodular map $\phi_f:S\loriar S$ such that
$\phi_f(x)=f(h^{-1}g)x$ for any $x\in p_{go} S p_{ho}$ and $g,h\in G.$
\item If $f$ is positive definite, then $\phi_f$ is \copo.
\item The restriction of $\phi_f$ to $S^G$ defines a trace-preserving $T^G$-bimodular map $\psi_f:S^G\loriar S^G$.
\item The map $\bar f$ has finite support if and only if $\psi_f(S^G)$ is a bifinite $T^G$-bimodule.
\end{enumerate}
\end{prop}

\begin{proof}
Proof of (1).
The map $\phi_f$ is well defined on the weakly dense $*$-subalgebra $B=\{x\in Gr\Pl\boxtimes Gr\Pl: \exists F\subset \lGGor \text{ finite such that } x=\sum_{s,t\in F} p_{so}xp_{to}\}.$
We have that
\begin{align*}
Tr\circ\phi_f(x)&=Tr(\sum_{g,h\in F} f(h^{-1}g)p_{go}xp_{ho})=\sum_{g\in F} Tr(f(g^{-1}g)p_{go}xp_{go})\\
&=\sum_{g\in F} Tr(p_{go}xp_{go})=Tr(x), \text{ for any } x\in B.
\end{align*}
Therefore, $\phi_f$ is trace-preserving on $B.$
The map $\phi_f$ is ultraweakly continuous and extends to a normal trace-preserving map $\phi_f:S\loriar S$.
This map is clearly unital since $f$ is unital.
Consider $t_1,t_2\in T$ and $x\in\sum_{g,h \in F}p_{go}Sp_{ho}$ where $F$ is a finite set.
Observe that $t_1p_{go} x p_{ho} t_2\in p_{go} S p_{ho}$ for any $g,h\in\lGGor.$
Therefore,
$$\phi_f(t_1xt_2)=\phi_f(\sum_{g,h\in F}t_1 p_{go} x p_{ho} t_2)=\sum_{g,h\in F}f(h^{-1}g)t_1 p_{go} x p_{ho} t_2=t_1\phi_f(x)t_2.$$
Hence, by a density argument we have that the map $\phi_f$ is $T$-bimodular.

Proof of (2).
Suppose that $f$ is \pode.
There exists a \HS\ $\mH$ and a continuous right-$G_o$-invariant bounded map $\xi:G\loriar\mH$ such that $f(h^{-1}g)=\langle \xi(h),\xi(g)\rangle$ for any $g,h\in G.$
Consider the \stre\ $\rho:S\loriar B(L^2(S,Tr)\otimes \mH)$ defined as $\rho(x)(a\otimes\zeta)=xa\otimes \zeta$ for any $x\in S$, $a\in L^2(S,Tr)$ and $\zeta\in\mH$.
Consider the bounded operator $A:L^2(S,Tr)\loriar  L^2(S,Tr)\otimes\mH,a\in p_{go} L^2(S,Tr)\longmapsto a\otimes \xi(g),$ for any $g\in\lGGor.$
One can check that $\phi_f(x)=A^*\rho(x)A$ for any $x\in S.$
Therefore, $\phi_f$ is \copo.

Proof of (3).
Consider the action $\sigma:G\loriar\Aut(S).$
Since $\phi_f$ does not depend on the choice of \SR\ $\lGGor$ we get that $\phi_f\circ\sigma_g=\sigma_g\circ\phi_f$ for any $g\in G.$
This implies that $\phi_f(S^G)$ is contained inside $S^G.$
The map $\psi_f$ is $T^G$-bimodular since it is the restriction of a $T$-bimodular map.
Consider the normal map $E^S_T:S\loriar S,x\longmapsto \sum_{g\in \lGGor}p_{go}xp_{go}$.
The unique \NFTS\ $tr$ of $S^G$ is the restriction of the map $\omega:S\loriar\C,x\longmapsto Tr(p_o)^{-1} Tr(xp_o)$.
It is easy to see that $\omega\circ E^S_T=\omega$ and $E^S_T\circ \phi_f=E^S_T.$
Therefore, $$tr\circ\psi_f(x)=\omega\circ\phi_f(x)=\omega\circ E^S_T\circ \phi_f(x)=\omega(x)=tr(x),$$ for any $x\in S^G.$

Proof of (4).
Since $G/G_o$ is an \AN\ subgroup, we have that $\bar f$ is finitely supported \IFF\ the support $supp(f)$ is contained in finitely many double cosets of $G_o\backslash G/G_o.$
Since $f$ is $G_o$-bi-invariant, we have that $supp(f)$ is stable by left and right multiplication by $G_o.$
Hence, there exists $E\subset\GGG$ such that $supp(f)=G_o\cdot E \cdot G_o.$
Observe that the norm closure of the image of $\psi_f$ in $L^2(S^G,tr)$ is isomorphic to $\bigoplus_{g\in E} H_g.$
Proposition \ref{prop:bimodules} implies the result.
\end{proof}

\begin{proof}[Proof of Theorem \ref{theo:introone} ]
Consider $(\Ga,\mu), G,G_o,\Pl$ as above such that $G$ is amenable. 
Denote by $\Ql=\Pl^G$ the \SPA\ equal to the \FPS\ under the action of $G$.
As observed in Section \ref{sec:amenability}, the subgroup $G_o<G$ is co-amenable since $G$ is amenable.
By Section \ref{sec:amenability}, there exists a sequence of positive definite $G_o$-bi-invariant maps $f_n:G\loriar\C, n\gO$ with support contained in a finite union of right $G_o$-cosets and that converges pointwise to $1$.
By Proposition \ref{prop:maps}, the collection $(\psi_n:=\psi_{f_n},n\gO)$ is a sequence of \copo\ $T^G$-bimodular maps from $S^G$ to $S^G$ such that $\psi_n(S^G)$ is a bifinite $T^G$-bimodule for any $n\gO.$
It is easy to see that $\lim_{n\rightarrow\infty}\Vert\psi_n(x)-x\Vert_2=0$ for any $x\in S^G.$
Therefore, $T^G\subset S^G$ is co-amenable.

Consider the inclusion $M^G\vee (M^\op)^G\subset T^G.$
The map $\Theta:T(o)^{G_o}\loriar S,x\longmapsto \sum_{s\in\lGGor}\sigma_s(x)$ realizes an isomorphism from $T(o)^{G_o}$ onto $T^G$ such that $\Theta(M(o)^{G_o})=M^G$ and $\Theta((M(o)^\op)^{G_o})=(M^\op)^G$.
Therefore, the inclusion $M^G\vee (M^\op)^G\subset T^G$ is isomorphic to $T(o)^{G_o\times G_o}\subset T(o)^{G_o}$ for the action $\sigma\otimes \sigma^\op:G_o\times G_o\loriar \Aut(M(o))\times \Aut(M(o)^\op)<\Aut(T(o)).$
The group $G_o$ is amenable since it is a closed subgroup of the amenable group $G$.
By the same observation made in Section \ref{sec:amenability}, we have that the subgroup $G_o\subset G_o\times G_o$ given by the diagonal inclusion is co-amenable since $G_o\times G_o$ is amenable.
Therefore, the inclusion $M^G\vee (M^\op)^G\subset T^G$ is co-amenable by \cite[Proposition 6]{Monod_Popa_co-moy}.
The composition of two co-amenable inclusions is co-amenable by \cite[Theorem 3.2.4.1]{Popa_correspondances}.
Therefore, $M^G\vee (M^\op)^G\subset S^G$ is co-amenable.
Proposition \ref{prop:fixedpoint} implies that this later inclusion is isomorphic to the \SEI\ of $\Ql.$
Therefore, the \SPA\ $\Ql$ is amenable.
\end{proof}

We end this section with an observation regarding the principal graph of $\Ql.$

\begin{prop}
Let $(\Ga,\mu), G, G_o,$ and $\Ql$ satisfying the assumptions of this section.
Moreover, assume that $G$ is contained in $\AutG$.
Then the \SPA\ $\Ql$ has finite depth if and only if the graph $\Ga$ is finite.
\end{prop}

\begin{proof}
Recall that a \SPA\ $\Ql$ has finite depth if and only if its \SEI\ is a finite index \SF.
The \SEI\ of $\Ql$ is isomorphic to $M^G\vee (M^\op)^G\subset S^G$ by Proposition \ref{prop:fixedpoint}.
By Proposition \ref{prop:bimodules}, the space $L^2(S^G,tr)$ is isomorphic to a sum of bifinite $T^G$-bimodules indexed by $G_o\backslash G/G_o.$

Suppose that $\Ga$ is infinite.
Observe that $f:G\loriar\mathbb N, g\longmapsto d(o,go)$ is $G_o$-bi-invariant map.
Moreover, $Im f=\{m\in\mathbb N : \exists w\in V^+, d(o,w)=m\}$ since $G$ acts transitively on $V^+.$
Therefore, $\vert G_o\backslash G/G_o\vert \geqslant \vert Im f\vert.$
Since $\Ga$ has infinitely many vertices and is locally finite, there exists vertices in $V^+$ that are arbitrary far from $o\in V^+.$
Hence, $Im f$ is infinite and $\vert G_o\backslash G/G_o\vert =\infty.$
This implies that $[S^G : M^G\vee (M^\op)^G]=[S^G : T^G] [T^G : M^G\vee (M^\op)^G]=\infty.$

Suppose that $\Ga$ is finite.
Then $\vert G_o\backslash G/G_o\vert\leqslant \vert G/G_o\vert=\vert V^+\vert<\infty.$
Hence, $[S^G : T^G]<\infty.$
The group $G_o$ is necessarily finite since $\Ga$ has a finite number of vertices and edges and $G_o$ is contained in $\Aut(\Ga)_o$.
As observed in the last proof, the inclusion $M^G\vee (M^\op)^G\subset T^G$ is isomorphic to $T(o)^{G_o\times G_o}\subset T(o)^{G_o}$ and $T(o)$ is a factor.
Hence, $[S^G:M^G\vee (M^\op)^G]=[S^G:T^G] [T^G:M^G\vee (M^\op)^G]=[S^G:T^G] [T(o)^{G_o}:T(o)^{G_o\times G_o}]=[S^G:T^G] \vert G_o\vert<\infty.$
This concludes the proof.
\end{proof}

\section{Description of some \SEI s via Hecke pairs}\label{sec:Heckepair}

\subsection{Crossed products and Hecke pairs}
\subsubsection{The ordinary action case}
Consider an inclusion of groups $H<G$ and assume that $G/H$ is countable.
We say that $(G,H)$ is a \Hepa\ if the subgroup $H<G$ is \AN, i.e. for any $g\in G$ the group $H\cap gHg^{-1}$ has finite index inside $H$ and $gHg^{-1}.$
We refer the reader to \cite{Delaroche_AP_Coset_spaces} for more details on \Hepa s and operator algebras.
To $(G,H)$ can be associated a \VNA\ $L(G,H)$ analogous to the group \VNA.
In this section, we define the cocycle action of a \Hepa\ on a tracial \VNA\ and its corresponding twisted \CrP.
This has been considered in the framework of ordinary action on $C^*$-algebras by Palma, see \cite{Palma_HP_I,Palma_HP_II}.
Our approach is a natural generalization of the construction given in \cite[Section 1]{Delaroche_AP_Coset_spaces}.
We first define this notion in the ordinary action case which has much simpler formulas.

Let $(A,\tau)$ be a finite \VNA\ with a \nofa\ tracial state.
Let $\gamma:G\loriar\Aut(A,\tau)$ be a trace-preserving action of $G$ on $A$.
Let $\CAGH$ be the space of functions $f:G\loriar A$ such that 
$f(hgk)=\gamma_h(f(g))$ for any $g\in G$, $h,k\in H$, and such that the induced functions $\overline f :G/H\loriar A$ is finitely supported.
Note that if $f\in\CAGH$ and $g\in G$, then $f(g)$ is fixed by $\sigma(H\cap gHg^{-1})$.
We define a multiplication and an involution $*$ on $\CAGH$ as follows.
$$f_1f_2(g)=\sum_{s\in\lGHr}f_1(s)\gamma_s(f_2(s^{-1}g)), \text{ for any } f_1,f_2\in\CAGH \text{ and } g\in G,$$
where $\lGHr$ is a \SR\ of $G/H$,
$$f^*(g)=\gamma_g(f(g^{-1})^*), \text { for any } f\in\CAGH\ \text{ and } g\in G.$$
The space $\CAGH$ endowed with those operations is a unital $*$-algebra.
We have an inclusion of $A^H$ inside $\CAGH$ given by the map $j:a\in A^H\longmapsto f_a\in\CAGH$ such that $f_a(g)=a$ if $g\in H$ and zero otherwise. 
We identify $A^H$ and $j(A^H).$
Consider the linear functional $\omega:\CAGH\loriar\C,f\longmapsto \tau(f(1))$.
Let $\pi=\gamma\otimes Ad(\lambda):G\loriar \Aut(A\ootimes B(\ell^2(G/H)))$ be the group action defined as
$\pi_g(a\otimes e_{h, k})=\gamma_g(a)\otimes e_{gh, gk},$
for any $g\in G$, $h,  k\in G/H$ and $a\in A$.
Denote by $Tr:B(\ell^2(G/H))_+\loriar \overline \R_+$ the usual \nofa\ semi-finite tracial weight on $B(\ell^2(G/H))$ that sends any minimal projection to $1$.

\begin{prop}\label{prop:genuine}
The linear functional $\omega$ is faithful and the $*$-algebra $\CAGH$ acts by bounded operators on $L^2(\CAGH,\omega)$ by left multiplication.
Let $\vNAGH$ be the \VNA\ generated by this left action that we call the \CrP\ of $A$ by $(G,H).$
We have an isomorphism $\varphi$ of \VNA s from $\vNAGH$ onto the \FPS\ $(A\ootimes B(\ell^2(G/H)))^G$ such that $\varphi(f)=\sum_{s,t\in\lGHr} \gamma_s ( f(t) )\otimes e_{ s , s t }$ for any $f\in\CAGH.$
In particular, 
$$\varphi(A^H)=\{\sum_{s\in \lGHr}\gamma_s(a)\otimes e_{s,s}:a\in A^H\}.$$
\end{prop}

\begin{proof}
If $f\in\CAGH$, then $\omega(ff^*)=\tau(ff^*(1))=\sum_{s\in\lGHr}\tau(f(s)f(s)^*).$
Hence, if $\omega(ff^*)=0$, then $f(s)=0$ for any $s\in\lGHr$ since $\tau$ is faithful.
This implies that $\omega$ is faithful.

Consider the \HS\ $\ell^2(G/H,A)$ of right-$H$-invariant functions $f:G\loriar L^2(A)$ such that $\sum_{s\in\lGHr} \tau ( f(s)f(s)^* )<\infty$ with the norm $\Vert f\Vert_2=\sqrt{\sum_{s\in\lGHr} \tau( f(s)f(s)^* )}$.
The \HS\ $L^2(\CAGH,\omega)$ is isometric to a subspace of $\ell^2(G/H,A)$.
We identify $L^2(\CAGH,\omega)$ and this isometric subspace.
Consider $r\in\lGHr$ and $\{h_1,\cdots,h_k\}\subset H$ such that $HrH$ is equal to the disjoint union of the right cosets $\bigcup_{i=1}^kh_irH.$
Consider $a\in A$ and the function $f\in\CAGH$ such that its support is contained in $HrH$ and $f(r)=a.$
Consider an element $\xi\in\CAGH\subset L^2(\CAGH,\omega).$
We have that
$$(f\xi)(g)=\sum_{s\in\lGHr} f(s) \ga_s(\xi(s^{-1}g))=\sum_{i=1}^k\ga_{h_i}(a)\ga_{h_ir}(\xi((h_ir)^{-1} g)),\  g\in G.$$
We fix $1\leqslant i\leqslant k.$
Consider the map $\eta_i:g\in G\longmapsto \ga_{h_i}(a)\ga_{h_ir}(\xi((h_ir)^{-1} g)).$
Note that $\eta_i\in \ell^2(G/H,A)$ and $f\xi=\sum_{j=1}^k\eta_j.$
We have that 
\begin{align*}
\Vert\eta_i\Vert_2^2&=\sum_{s\in\lGHr}\tau(\ga_{h_i}(a)\ga_{h_ir}(\xi((h_ir)^{-1} s)\ga_{h_ir}(\xi((h_ir)^{-1} s)^*\ga_{h_i}(a)^*)\\
&=\sum_{s\in\lGHr}\tau(a^*a \ga_r(\xi((h_ir)^{-1} s)\xi((h_ir)^{-1} s)^*))\\
&\leqslant \Vert a\Vert_A^2\sum_{s\in\lGHr}\tau(\xi((h_ir)^{-1} s)\xi((h_ir)^{-1} s)^*)=\Vert a\Vert_A^2\Vert \xi\Vert_2^2,
\end{align*}
where $\Vert\cdot\Vert_A$ is the C$^*$-norm of $A.$
Hence, $\Vert f\xi\Vert_2\leqslant k\Vert a\Vert_A \Vert \xi\Vert_2.$
Since the subspace $\CAGH\subset L^2(\CAGH,\omega)$ is dense, we obtain that the left multiplication by $f$ on $L^2(\CAGH,\omega)$ is bounded.
Therefore, the left multiplication of $\CAGH$ defines a $*$-representation of $\CAGH$ on the GNS \HS\ $L^2(\CAGH,\omega)$.
It is clear the $\varphi$ defines a normal $*$-morphism from $\vNAGH$ to the \FPS\ $(A\ootimes B(\ell^2(G/H)))^G$.
If $\varphi(f)=0$, then $\gamma_s(f(t))=0$ for any $s,t\in\lGHr.$
Hence, $f=0$ and $\varphi$ is injective.
Consider $x=\sum_{s,t\in G/H} x_{s,t}\otimes e_{s,t}\in (A\ootimes B(\ell^2(G/H)))^G$.
Since $x$ is $G$-invariant, we have that $x_{gs,gt}=\ga_g(x_{s,t})$ for any $g\in G,s,t\in G/H.$
Therefore, $x=\varphi(f)$ where $f(t)=x_{1,t}, t\in \lGHr.$
Hence, $\varphi$ is surjective. We obtain that $\varphi$ realizes an isomorphism of \VNA s.
The last assertion of the proposition follows from the definition of $\varphi.$
\end{proof}

\subsubsection{The twisted case}
Let $G,H,A,\tau$ be as above.
A cocycle action of the \Hepa\ $(G,H)$ on the tracial \VNA\ $(A,\tau)$ is a couple of maps $(\gamma,u)$ where 
$$\ga:G\times G/H\loriar \Aut(A,\tau) \text{ and } u:G\times G/H\times G/H\loriar U(A)$$
satisfying the following axioms:
\begin{enumerate}
\item $\ga_{1,s}=\text{Id}$,
\item $\ga_{gh,s}=\ga_{g,hs}\circ\ga_{h,s}$,
\item $\ga_{g,s}=\text{Ad}(u_{g,s,t})\circ \ga_{g,t},$
\item $u_{1,s,t}=u_{g,s,s}=1,$
\item $u_{g,s,t} u_{g,t,r}=u_{g,s,r}$, and
\item $u_{gh,s,t}=\ga_{g,hs}(u_{h,s,t})u_{g,hs,ht}$ for any $g,h \in G$ and $s,t,r\in G/H.$
\end{enumerate}
We continue to denote by $\CAGH$ the space of functions $f:G\loriar A$ such that 
$f(hgk)=\gamma_{h,1}(f(g))u_{h,1,g}$ for any $g\in G$, $h,k\in H$, and such that the induced functions $\overline f :G/H\loriar A$ is finitely supported.
We define a multiplication and an involution $*$ on $\CAGH$ as follows.
$$f_1f_2(g)=\sum_{s\in\lGHr}f_1(s)\ga_{s,1}(f_2(s^{-1}g))u_{s,1,s^{-1}g}, \text{ for any } f_1,f_2\in\CAGH \text{ and } g\in G,$$
where $\lGHr$ is a \SR\ of $G/H$,
$$f^*(g) = \ga_{g,g^{-1}} ( f ( g^{-1} )^* ) u_{ g , g^{-1} , 1 }, \text { for any } f\in\CAGH\ \text{ and } g\in G.$$
A careful check shows that the space $\CAGH$ endowed with those operations is a unital associative $*$-algebra.
Observe that the map $h\in H\longmapsto \gamma_{h,1}\in\Aut(A,\tau)$ is a group morphism. 
Hence, we have an ordinary action of $H$ on $A$.
Denote by $A^H$ the algebra of $a\in A$ such that $\gamma_{h,1}(a)=a$ for any $h\in H.$
We have an inclusion of $A^H$ inside $\CAGH$ given by the map $j:a\in A^H\longmapsto f_a\in\CAGH$ such that $f_a(g)=a$ if $g\in H$ and zero otherwise. 
Consider the linear functional $\omega:\CAGH\loriar\C,f\longmapsto \tau(f(1))$.

\begin{prop}\label{prop:twisted}
The linear functional $\omega$ is faithful and the $*$-algebra $\CAGH$ acts by bounded operators on $L^2(\CAGH,\omega)$ by left multiplication.
Let $\vNAGH$ be the \VNA\ generated by this left action that we call the twisted \CrP\ of $A$ by $(G,H).$
There exists a unique group morphism $\pi:G\loriar\Aut(A\ootimes B(\ell^2(G/H)))$ satisfying that $\pi_g(a\otimes e_{s,t})=\gamma_{g,s}(a)u_{g,s,t}\otimes e_{gs,gt}$ for any $a\in A, g\in G, s,t\in G/H.$
We have an isomorphism $\varphi$ of \VNA s from $\vNAGH$ onto the \FPS\ $(A\ootimes B(\ell^2(G/H)))^G$ such that $\varphi(f)=\sum_{s,t\in\lGHr} \ga_{s,1} ( f(t) )u_{s,1,t}\otimes e_{ s , s t }$ for any $f\in\CAGH.$
In particular, 
$$\varphi(A^H)=\{\sum_{s\in \lGHr}\ga_{s,1}(a)\otimes e_{s,s}:a\in A^H\}.$$
\end{prop}
\begin{proof}
The proof is similar to the proof of Proposition \ref{prop:genuine}.
\end{proof}

\begin{remark}\label{rem:CP}
If $A=\C$, then the \VNA\ $\vNAGH$ is the \VNA\ $L(G,H)$ (possibly twisted by a cocycle) associated to the \Hepa\ $(G,H)$ \cite{Delaroche_AP_Coset_spaces}.
If $H<G$ is a normal subgroup, then the \VNA\ $\vNAGH$ is isomorphic to a (twisted) \CrP\ $A^H\rtimes (G/H)$ of the \FPS\ $A^H$ by the quotient group $G/H.$
If the standard \SMU\ of $B(\ell^2(G/H))$ is $G$-invariant, i.e. $\pi_g(1\otimes e_{s,t})=1\otimes e_{gs,gt}$ for any $g\in G, s,t\in G/H$, then $ (A\ootimes B(\ell^2(G/H)))^G$ is isomorphic to $\vNAGH$ for an ordinary action of $(G,H)$ on $A.$
\end{remark}

\subsection{Description of \SEI s}
Let $(\Ga,\mu)$ be a weighted graph with vertex set $V=V^+\cup V^+$ and let $\Pl$ be its associated \BGPA.
Moreover, assume that $\mu$ is constant on the set of edges with source in $V^+.$
Let us fix a vertex $o\in V^+$ and denote by $\TS$ the \SEI\ of $\Pl.$
The set $\{e_v:v\in V^+\}$ is the set of minimal projections of $\ell^\infty(V^+)$ and $\{e_{v,w}:v,w\in V^+\}$ is the standard \SMU\ of $B(\ell^2(V^+)).$
We continue to denote by $T(v)$ the corner $Tp_v$ for any $v\in V^+.$

\begin{theo}\label{theo:SEIP}
The \SEI\ of $\Pl$ is isomorphic to $$T(o)\ootimes \ell^\infty(V^+)\subset T(o)\ootimes B(\ell^2(V^+)),$$
where the inclusion is given by $a\otimes e_v\mapsto a\otimes e_{v,v}$ for any $a\in T(o),v\in V^+.$
\end{theo}

\begin{proof}
By Proposition \ref{prop:vna}, $Tr(p_v)=Tr(p_o)$ for any $v\in V^+.$
Since $S$ is a factor, there exists a \SMU\ $\{\ep_{v,w}:v,w\in V^+\}\subset S$ such that $\ep_{v,w}\ep_{v,w}^*=p_v$ and $\ep_{v,w}^*\ep_{v,w}=p_w$ for any $v,w\in V^+.$
Consider $v\in V^+$ and the map $\beta_v:T(o)\loriar S, x\longmapsto \ep_{v,o} x \ep_{o,v}$.
Since $p_vSp_v=T(v)$, we have that the range of $\beta_v$ is contained in $T(v).$
Observe that $\beta_v$ realizes an isomorphism of \VNA s from $T(o)$ onto $T(v)$ and its inverse is the map $\beta_v^{-1}: T(v)\loriar T(o),x\longmapsto \ep_{o,v} x \ep_{v,o}$.
Observe that $S$ is the weak closure of the vector space 
$\spann(x \ep_{v,w}: v,w\in V^+, x\in T(v)).$
Consider the densely defined map $\phi:S\loriar T(o)\ootimes B(\ell^2(V^+))$ defined as 
$\phi(x \ep_{v,w})=\beta_v^{-1}(x)\otimes e_{v,w}$ for any $v,w\in V^+,x\in T(v).$
This map is clearly a $*$-morphism and extends to a normal map from $S$ to $T(o)\ootimes B(\ell^2(V^+))$.
Since $S$ is a factor and $\phi$ is not identically equal to zero we have that $\phi$ is injective.
The range of $\phi$ is weakly dense since it contains the set $\{x\otimes e_{v,w}: x\in T(o),v,w\in V^+\}$.
This implies that $\phi$ is surjective.
Therefore, $\phi$ realizes an isomorphism of \VNA s.
If $x\in T$, then $\phi(x)=\sum_{v\in V^+} \beta_v^{-1}(xp_v)\otimes e_{v,v}$.
This implies that $\phi(T)$ is contained in $T(o)\ootimes \ell^\infty(V^+)$.
Furthermore, $\phi(T)$ is clearly dense in $T(o)\ootimes \ell^\infty(V^+).$
Therefore, $\phi(T)=T(o)\ootimes\ell^\infty(V^+).$
\end{proof}

Consider a subgroup $G<\AutP$ such that $\Ql:=\Pl^G$ is a \SPA.
In particular, $G$ acts transitively on $V^+$.
We identify $V^+$ and $G/G_o$.
Let $\sigma:G\loriar \Aut(S)$ be the action defined in Section \ref{sec:fixedpoint}.
Note that we have an action of $G_o\times G_o$ on the corner $T(o)$ given by $(g,h)\cdot a\otimes b^\op=\sigma_g(a)\otimes \sigma_h(b)^\op$ for any $g,h\in G_o,a,b\in M(o).$
We denote by $T(o)^{G_o\times G_o}$ the \FPS\ with respect to this action.
Consider the diagonal inclusion of groups $g\in G_o\longmapsto (g,g)\in G_o\times G_o$ and the corresponding inclusion of \VNA s $T(o)^{G_o\times G_o}\subset T(o)^{G_o}.$
Denote by $\Ql=\Pl^G$ the \FPS\ of $\Pl$ with respect to the action of $G.$

\begin{theo}\label{theo:SEIQ}
There exists a cocycle action $(\ga,u)$ of the \Hepa\ $(G,G_o)$ on the II$_1$ factor $T(o)$ such that $\ga_{g,1}(a\otimes b^\op)=\sigma_g(a)\otimes \sigma_g(b)^\op$ for any $g\in G, a,b\in M(o)$, and such that the \SEI\ of $\Ql$ is isomorphic to 
$$T(o)^{G_o\times G_o}\subset \text{vN}[T(o);G,G_o].$$
\end{theo}

\begin{proof}
The map $g\in G\longmapsto go\in V^+$ is surjective since the action is transitive.
It realizes a bijection from $G/G_o$ onto $V^+.$
Denote by $\lGGor$ a \SR\ of $G/G_o.$
By Theorem \ref{theo:SEIP}, the \SEI\ of $\Pl$ is isomorphic to $T(o)\ootimes \ell^\infty(V^+)\subset T(o)\ootimes B(\ell^2(V^+))$ via the map $\phi:S\loriar T(o)\ootimes B(\ell^2(V^+))$ that satisfies $\phi(a\ep_{v,w})=\ep_{o,v} a\ep_{v,o}\otimes e_{v,w}$ for any $v,w\in V^+, a\in T(v).$
We have that the group action 
$$\pi:G\loriar \Aut(T(o)\ootimes B(\ell^2(V^+)),g\longmapsto \phi\circ\sigma_g\circ\phi^{-1}$$ 
satisfies that
\begin{equation}\label{equa:pi}
\pi_g(T(o)\otimes e_{v,w})=T(o)\otimes e_{gv,gw} \text{ and } \pi_g(1\otimes e_{v,v})=1\otimes e_{gv,gv} \text{ for any } g\in G, v,w \in V^+.
\end{equation}
This implies that for any $g\in G, v,w\in V^+$ there exists a unitary $u_{g,v,w}\in U(T(o))$ such that $\pi_g(1\otimes e_{v,w})=u_{g,v,w}\otimes e_{gv,gw}.$
Furthermore, for any $a\in T(o), v\in V^+, g\in G$ there exists $\ga_{g,v}(a)\in T(o)$ such that $\pi_g(a\otimes e_{v,v})=\ga_{g,v}(a)\otimes e_{gv,gv}.$
Since $\pi_g$ is an automorphism and $\pi_g(T(o)\otimes e_{v,v})=T(o)\otimes e_{gv,gv}$ we necessarily have that $\ga_{g,v}$ is an automorphism of $T(o)$ for any $g\in G, v\in V^+.$
From \eqref{equa:pi} and the fact that $\pi$ is a group action we can deduce that $(\ga,u)$ is a cocycle action of $(G,G_o)$ on the II$_1$ factor $T(o).$
Let $\vNTGGo$ be the \CrP\ \VNA\ associated to the cocycle action $(\ga,u)$ of the \Hepa\ $(G,G_o)$ on the II$_1$ factor $T(o).$
By Proposition \ref{prop:twisted}, we obtain that the inclusion $T(o)^{G_o}\subset \vNTGGo$ is isomorphic to the inclusion \begin{equation}\label{equa:SEI}
\{\sum_{s\in\lGGor} \sigma_s(x) \otimes e_{so,so}:x\in T(o)^{G_o}\}\subset (T(o)\ootimes B(\ell^2(V^+))^G,
\end{equation}
where $(T(o)\ootimes B(\ell^2(V^+))^G$ is the \FPS\ under the action of $\pi$.
Consider the isomorphism $\phi:S\loriar T(o)\ootimes B(\ell^2(V^+)$ of the proof of Theorem \ref{theo:SEIP}.
The \SEI\ of $\Ql$ is isomorphic to $M^G\vee (M^\op)^G\subset S^G$ by Proposition \ref{prop:fixedpoint}.
We have that $\phi(S^G)=(T(o)\ootimes B(\ell^2(V^+))^G$.
Observe that $\phi(M^G)=\{\sum_{s\in\lGGor} \sigma_s(a)\otimes 1\otimes e_{so,so}:a\in M(o)^{G_o}\}$ and $\phi((M^\op)^G)=\{\sum_{s\in\lGGor} 1\otimes (\sigma_s(b))^\op\otimes e_{so,so}:b\in M(o)^{G_o}\}$.
This last observation together with the characterization of the image of $T(o)^{G_o}$ given in \eqref{equa:SEI} imply that the \SEI\ of $\Ql$ is isomorphic to the inclusion $T(o)^{G_o\times G_o}\subset \text{vN}[T(o);G,G_o].$
\end{proof}

\subsection{Examples}\label{sec:examples}

We present examples of \SPA s for which we can describe their \SEI\ via our last theorem.

\subsubsection{Diagonal subfactors}

The \SEI\ of a diagonal \SF\ is known to be isomorphic to the (twisted) \CrP\ of a II$_1$ factor by a group \cite[Section 3]{Popa_symm_env_T}.
We give here a new proof of this fact which is a particular case of Theorem \ref{theo:SEIQ}.
Consider a II$_1$ factor $L$ and a finite set $\{g_1,\cdots,g_n\}$ of outer automorphisms of $L$.
Denote by $G$ the subgroup of the outer automorphism group Out$(L)$ generated by $\{g_1,\cdots,g_n\}$.
Let $\NM$ be the \SF\ equal to the inclusion $\{ \sum_{i=1}^{n+1} \alpha_{g_i}(x)\otimes e_{i,i}:x\in L\}\subset L\otimes \mathcal M_{n+1}(\C)$, where $\{e_{i,j}:i,j=1,\cdots n+1\}$ is the usual \SMU\ of the type I$_{n+1}$ factor $\mathcal M_{n+1}(\C)$ and $g_{n+1}=1$.
This subfactor is called a diagonal subfactor.
By \cite{Bisch_Das_Ghosh_Diagonal,Burstein_BGPA}, the \SPA\ of $\NM$ can be described as the fixed point \PA\ of a \BGPA\ as follows.
Consider the bipartite graph $\Ga$, where $V^+=\{v^+_g:g\in G\}$ is a copy of $G$ and $V^-=\{v^-_g:g\in G\}$ is another copy of $G$.
For any $g\in G$ and $1\leqslant i\leqslant n+1$ there is an edge from $v_g^+$ to $v^-_{gg_i}.$
Let $\mu$ be the weight of $\Ga$ that assigns $1$ to any edge and let $\Pl$ be the \BGPA\ associated to $(\Ga,\mu).$
The group $G$ acts on the weighted graph $(\Ga,\mu)$ by left multiplication.
The \FPPA\ $\Ql=\Pl^G$ is isomorphic to the \SPA\ of $\NM$.
For any even vertex of $\Ga$, the subgroup of $G$ that fixes this vertex is trivial.
Therefore, by Theorem \ref{theo:SEIQ} and Remark \ref{rem:CP}, there exists a II$_1$ factor $A$ and a trace-preserving cocycle action of $G$ on $A$ such that the \SEI\ of $\Ql$ is isomorphic to $A\subset A\rtimes G.$

\subsubsection{Bisch-Haagerup subfactors}

Consider a II$_1$ factor $L$ and two finite subgroups $H,K\subset \text{Out}(L).$
Assume that the intersection $H\cap K$ is the trivial group.
Denote by $G$ the subgroup of Out$(L)$ generated by $H$ and $K.$
Assume that there exists a group morphism $\alpha:G\loriar Aut(L)$ which is a lift of the inclusion $G\subset Out(L).$
The subfactor $L^K\subset L\rtimes H$ is called a \biha-\SF\ \cite{Bisch_Haagerup_composition_subfactors}.
By \cite{Bisch_Das_Ghosh_BH_pa,Burstein_BGPA}, its \SPA\ is isomorphic to the \FPS\ of a \BGPA\ as follows.
Consider the bipartite graph $\Ga$ where $V^+=\{v_{gH}: gH\in G/H\}$ is a copy of the coset space $G/H$ and $V^-=\{v_{gK}:gK\in G/K\}$ is a copy of the coset space $G/K$.
The set of edges $C_1=\{e_g:g\in G\}$ is equal to a copy of $G$ where $e_g$ is an edge between $v_{gH}$ and $v_{gK}$.
Let $\mu$ be the weight on $\Ga$ that assigns $\vert K\vert^{1/2}/\vert H\vert^{1/2}$ to any edge in $C_1^+.$ 
Let $\Pl$ be the \BGPA\ associated to $(\Ga,\mu).$
The group $G$ acts by left multiplication on $(\Ga,\mu)$ and the \FPS\ $\Ql=\Pl^G$ is isomorphic to the \SPA\ of the \SF\ $L^K\subset L\rtimes H$.
Consider the even vertex $v_H$.
The subgroup of $G$ that fixes $v_H$ is equal to $H$.
By Theorem \ref{theo:SEIQ}, there exists a II$_1$ factor $A$ and a cocycle action of the \Hepa\ $(G,H)$ on $A\ootimes A^\op$ such that $H$ acts on $A$ and such that the \SEI\ of $\Ql$ is isomorphic to $A^H\ootimes (A^\op)^H\subset \text{vN}[A\ootimes A^\op;G,H].$

\subsubsection{Binary trees and subfactors}

Consider two non-zero natural numbers $r^+$ and $r^-$.
Let $\Ga=T(r^\pm)$ be the infinite bipartite tree where any vertex $v\in V^\pm$ has degree $r^\pm.$
Define the weight $\mu$ such that $\mu(a)=\sqrt{r^-/r^+}$ for any $a\in C_1^+.$
We have that the weighted graph $(\Ga,\mu)$ satisfies all the assumptions of Section \ref{sec:preliminaries} and has modulus $\delta$ equal to $\sqrt{r^+r^-}.$
Let $\Pl$ be the \BGPA\ associated to $(\Ga,\mu).$
The group $\Aut(\Ga,\mu)$ is equal to the automorphism group of the bipartite graph $G=\Aut(\Ga)$ which acts transitively on the even and odd vertices and on the edges.
Let $o$ be an even vertex of $\Ga$ and $G_o$ the subgroup of $G$ that fixes this vertex.
The group $G_o$ is an infinite group.
Let $\Ql$ be the \FPPA\ under the action of $G$.
It is an irreducible \SPA.
By Theorem \ref{theo:SEIQ}, there exists a II$_1$ factor $A$ and a cocycle action of the \Hepa\ $(G,G_o)$ on $A\ootimes A^\op$ such that $G_o$ acts on $A$ and such that the \SEI\ of $\Ql$ is isomorphic to $A^{G_o}\ootimes (A^\op)^{G_o}\subset \text{vN}[A\ootimes A^\op;G,G_o].$

\subsection*{Acknowledgement}
Part of this work was done when the author was visiting the Institute Monsieur Matthieu in Paris during Summer 2015.
He gratefully acknowledges the kind hospitality he received.
The author expresses his gratitude to Cyril Houdayer and Jesse Peterson for very valuable comments and for constant support and encouragement of Dietmar Bisch and Vaughan Jones.
The author was partially supported by NSF Grant DMS-1362138.

\bibliographystyle{alpha}

\end{document}